\documentclass[11pt,a4paper]{article}
\usepackage{amssymb,amsfonts,amsmath,amsthm, mathtools}
\usepackage[latin1]{inputenc}
\usepackage[english]{babel}
\usepackage[left=2.7cm,right=2.7cm,bottom=2.67cm,top=2.67cm]{geometry}
\usepackage{dsfont}
\usepackage{mathrsfs}
\usepackage{natbib}
\usepackage{xcolor}
\usepackage{graphicx}
\usepackage{multirow}
\usepackage{multicol}
\usepackage{placeins}
\usepackage[normalem]{ulem}

\usepackage[implicit=false]{hyperref}

\theoremstyle{plain}
\newtheorem{thm}{Theorem}
\newtheorem{prop}{Proposition}
\newtheorem{lemma}{Lemma}
\newtheorem{coro}{Corollary}
\renewcommand{\qed}{\hfill \mbox{\raggedright \rule{.07in}{.1in}}}

\DeclarePairedDelimiter\abs{\lvert}{\rvert}

\DeclarePairedDelimiter\parent{(}{)}
\DeclarePairedDelimiter\colc{[}{]}
\DeclarePairedDelimiter\chave{\{}{\}}

\graphicspath{{figs/}}
\theoremstyle{definition}
\newtheorem{remark}{Remark}
\newtheorem{df}{Definition}

\newcommand{\F}{\mathscr{F}}
\newcommand{\R}{\mathds{R}}
\newcommand{\N}{\mathds{N}}

\def\A{\mathcal{A}}

\renewcommand{\qed}{\hfill \mbox{\raggedright \rule{.07in}{.1in}}}

\newcommand{\var}{\mathrm{Var}}
\DeclareMathOperator*{\cov}{Cov}

\DeclareMathOperator*{\argmin}{arg\,min}

\newcommand{\oac}{o_{a.co.}}
\newcommand{\Oac}{O_{a.co.}}
\newcommand{\supp}{\mathrm{supp}}

\DeclareRobustCommand{\rchi}{{\mathpalette\irchi\relax}}
\newcommand{\irchi}[2]{\raisebox{\depth}{$#1\chi$}}

\DeclareRobustCommand{\rphi}{{\mathpalette\irphi\relax}}
\newcommand{\irphi}[2]{\raisebox{\depth}{$#1\phi$}}

\DeclareRobustCommand{\rgamma}{{\mathpalette\irgamma\relax}}
\newcommand{\irgamma}[2]{\raisebox{\depth}{$#1\gamma$}}

\long\def\sfootnote[#1]#2{\begingroup%
\def\thefootnote{\fnsymbol{footnote}}\footnote[#1]{#2}\endgroup}
\def\bfootnote{\xdef\@thefnmark{}\@footnotetext}

\begin{document}
\pagestyle{myheadings} 
\markboth{Strong consistency of the local linear estimator with dependent functional data}{Matsuoka and Torrent}

\thispagestyle{empty}
{\centering
\Large{\bf Strong consistency of the local linear estimator for a generalized regression function with dependent functional data.} \vspace{.5cm}\\
\normalsize{{\bf 
Danilo H. Matsuoka${}^{\mathrm{a,}}$\sfootnote[1]{Corresponding author. This Version: \today},\let\thefootnote\relax\footnote{\hskip-.3cm$\phantom{s}^\mathrm{a}$Research Group of Applied Microeconomics - Department of Economics, Federal University of Rio Grande.} Hudson da Silva Torrent${}^\mathrm{b}$\let\thefootnote\relax\footnote{\hskip-.3cm$\phantom{s}^\mathrm{b}$Mathematics and Statistics Institute - Universidade Federal do Rio Grande do Sul.
}
 \\
\let\thefootnote\relax\footnote{E-mails: danilomatsuoka@gmail.com (Matsuoka);  hudsontorrent@gmail.com (Torrent)}
\vskip.3cm
}
}
}
\begin{abstract}
In this study, we focus on a generalized nonparametric scalar-on-function regression model for heterogeneously distributed and strongly mixing  data. We provide almost complete convergence rates for the local linear estimator of the regression function. We show that, under our conditions, the pointwise and uniform convergence rates are the same on a compact set. On the other hand, when the data is dependent, it is proved that the convergence rate  can be slower than those  obtained for independent data. A simulation study shows the good performance and finite  sample properties of the
	functional local linear estimator (FLL) in comparison to the local constant estimator (FLC). In	addition, a one step ahead energy consumption forecasting exercise illustrates that the forecasts of the FLL estimator are significantly more accurate than those of the FLC.
\vspace{.2cm}\\
\noindent \textbf{Keywords:}  Almost complete convergence; Local linear estimator; Functional data; Mixing; Nonparametric regression; Asymptotic theory.\vspace{.2cm}\\
\noindent \textbf{MSC2020:} 62G20, 62G08, 	62R10.
\end{abstract}
\section{Introduction}
Popularized by \cite{ferraty}, the nonparametric approach in functional regression models has been studied intensively in the last years. To cite a few papers, the local constant estimator (also known as the \textit{Nadaraya-Watson}  estimator) or its variations have been employed to estimate the nonparametric regression function \citep{laib,ling,zhu_politis,kara,shang}, the conditional density \citep{ezzahrioui,liang,liang2} and the conditional distribution function \citep{horrigue}. 

In most situations, the model under investigation involves a scalar response and  a functional covariate. However, some works provided results for models where the response variable is also functional \citep{lian} or multivariate \citep{omar}. 

As in the finite dimensional setting, the Nadaraya-Watson estimator is a particular case of a wider class of kernel-based estimators called \textit{local polynomial regression estimators} \citep[see][]{wand}. The latter is constructed assuming that the regression function is locally well approximated by a polynomial of a given order $k\in\N$ whereas the former fixes $k=0$. The local linear estimator ($k=1$) becomes popular due to its desirable properties\footnote{It does not suffer from boundary bias and adapts to both random and fixed designs \citep[see][]{fan,wand}.} and its relative simplicity.  \cite{baillo}, \cite{berlinet} and \cite{barrientos} were the first to propose adaptations of the local linear ideas to functional data. It should be noted that the precursor work of \cite{barrientos} has influenced the development of several subsequent contributions, including the estimation of the conditional density \citep{demongeotdensity} and the conditional distribution function \citep{demongeotcdf,messaci}, the asymptotic normality for independent \citep{zhou_lin} and dependent \citep{xiong} data, the estimation for censored data  \citep{leulmi2}, an  estimation robust to outliers and heteroskedasticity \citep{belarbi} among others. 

We highlight the extension made by \cite{leulmi}  to provide strong convergence rates for strongly mixing functional data. We modify their set of assumptions in order to accommodate usual asymmetric kernel functions like the polynomial-type kernels (e.g., triangle, quadratic, cubic, and so on) and to allow for a more general dependence condition. With regard to the latter aspect, we weaken the conditions on the relation between joint probabilities and products of small ball probabilities  for strongly mixing data. Here, the data are allowed to be heterogeneously distributed.

The aim of this investigation is to study the almost complete convergence of the local linear estimator, pointwise and uniformly, for functional data under strong mixing dependence. As mentioned above, \cite{leulmi} has already investigated a similar problem. However, their asymptotics is developed slightly differently. 

The remainder of this paper is organized as follows. In Sec. \ref{sec2}, some preliminary definitions and notation are introduced. In Sec. \ref{sec3}, a list of assumptions \textcolor{purple}{is} given and the convergence rates of the local linear estimator are established. Sec. \ref{sec4} presents a simulation experiment, and Sec. \ref{sec5} complements the study with an application to energy consumption data. In Sec. \ref{sec6}, a global conclusion is given. The proofs of our main results and lemmas are presented in Appendices A and B, respectively.

\section{Model and estimation}\label{sec2}

To formulate the estimation problem, introduce $n$ random pairs  $(Y_i,\rchi_i)$, $i\in\{1,\dotsc,n\}$,  on $(\Omega,\A,P)$  taking values in $\R\times \F$, where $\F$ is some abstract semimetric space\footnote{In this work, a semimetric $d$ is defined as in  Definition 3.2 of \cite{ferraty}.   In some fields of Mathematics, especially in Topology, $d$ is better known as a \textit{pseudometric} \citep[see][]{kelley,howes}.} equipped with a semimetric $d$.  Furthermore, suppose that each pair $(Y_i,\rchi_i)$ follows the generalized regression model:
\begin{equation}\label{eq1}
	\varphi(Y_i)=m_\varphi(\rchi_i)+\epsilon_i, \quad i\in\N,
\end{equation}
where $m_\varphi:\F\to\R$ is called the \textit{regression function}, $\varphi:\R\to\R$ is a Borel function and the random error $\epsilon_i$ is such that $E(\epsilon_i)=0$ and is independent of $\rchi_j$ for all $i\neq j$. Note that $\{(Y_i,\rchi_i)\}_{i=1}^n$ is allowed to be dependent and heterogeneously distributed.

It is clear that $(1)$ is a generalization of the standard regression model in the extent that $\varphi$ can be set as the identity function (i.e. $\varphi(t)=t$). 

Indeed, the above generalized model encompasses a broad set of  nonparametric estimation problems. For example, the conditional  cumulative distribution function (c.d.f.) can be studied by setting $\varphi(t)=1_{(-\infty,y]}(t)$,  for any $y\in\R$, because then $m_\varphi(x)=P(Y\leq y\mid \rchi=x)$. Under some regularity conditions \cite[see][]{demongeotcdf},  if instead $\varphi(t)= H(\textcolor{purple}{(}y-t)/h_n)$ where $H$ is some c.d.f. and $h_n=o(1)$, then  $m_\varphi (x)\to P(Y\leq y\mid \rchi=x)$ as $n\to\infty$ . On the other hand, when one is interested in the conditional density  $f_{Y\mid \rchi}$ (assuming it exists and is smooth enough), the choice of $\varphi(t)=G((y-t)/h_n), y\in\R$, with $G$ being a kernel function implies that $m_\varphi (x)\to f_{Y\mid\rchi}(y\mid x)$ as $n\to\infty$  \citep[see][]{demongeotdensity}.

As proposed by \cite{barrientos}, a \textit{local linear estimator}  $\hat m_{\varphi}(x)$ for the regression function $m_\varphi(x)=E(\varphi(Y_i)\mid\rchi_i=x),\ x\in\F,$  can be defined as the solution $a$ of the following minimization problem
\begin{equation}\label{eq2}
	\min_{(a,b)\in\R^2}\sum_{i=1}^n [\varphi(Y_i)-a-b\beta(\rchi_i,x)]^2K(d(\rchi_i,x)/h),
\end{equation}
where $\beta:\F^2\to\R$ is a known function such that, $\forall x'\in\F, \ \beta(x',x')=0$,  $\{h\}\coloneqq \{h_n\}$ is a strictly positive sequence satisfying $h=o(1):nh\to\infty$, as $n\to\infty$, and the function $K:\R\to\R_+$ is a known asymmetrical kernel function with $\R_+$ denoting the set of nonnegative real numbers. It can be shown that \eqref{eq2} admits the explicit formula for $a$:
\begin{equation}\label{eq3}
	\hat m_\varphi (x)=\frac{\sum_{i,j=1}^n w_{i,j}(x) \varphi_j}{\sum_{i,j=1}^n w_{i,j}(x)},
\end{equation}
with 
\begin{equation*}
	w_{i,j}(x)=\beta_i(x)(\beta_i(x)-\beta_j(x))K_i(x)K_j(x),
\end{equation*}
where, by a slight abuse of notation, $K_i(x)=K(d(\rchi_i,x)/h)$, $\beta_i(x)=\beta(\rchi_i,x)$ and $\varphi_i=\varphi(Y_i)$.

The estimator in \eqref{eq3} is motivated by the assumption that  $a-b\beta(\cdot,x)$ is a good approximation of $m_\varphi(\cdot)$ around $x$. It implies that $a$ is approximately  $m_\varphi(x)$ as $\beta(x,x)=0$,  leading to the idea that $m_\varphi(x)$ could be reasonably estimated by $\hat m_\varphi (x)$. Clearly, the estimator is conceptually a local weighted least squares with kernel weights $K_i$. This approach is a natural extension to that of used in the traditional multivariate local polynomial regression\footnote{For further details, see Section 5.2 of \cite{wand} and Section 1.6 of \cite{tsybakov}.} where the regression function is approximated by its Taylor polynomial of some degree at $x$.

The functions $\beta(\cdot,\cdot)$ and $d(\cdot,\cdot)$  can be regarded as  locating functions as they locate one element of $\F$ with respect to another element in $\F$. While $\beta(\cdot,x)$ is determined by the hypothesis on how $a-b\beta(\cdot,x)$ fits the data $\{(Y_i,\rchi_i)\}$ near $x$, the semimetric $d(\cdot,x)$ is more related to the  topological structure of $\F$ which also affects the weighting scheme in \eqref{eq2}. Theoretically, the semimetric $d$ plays a central role in the quality of the convergence of kernel estimators since it controls the behavior of small ball probabilities around zero \citep[see Chapter 13 of][]{ferraty}.  The bandwith  $h$ can be regarded as a smoothing parameter, where larger values of $h$ tend to weight the observations more equally.

\section{Asymptotics}\label{sec3}

\subsection{Preliminaries}\label{sec3.1}

Some preliminary concepts are needed for our asymptotics. For  easy reference, consider the following definitions.

\begin{df}[Strong mixing]\label{d1}
	Let $\{X_i\}_{i\in\N}$ be a sequence of random variables and  let $\mathcal F_\ell^m=\sigma(X_i:\ell\leq i\leq m)$ be the sigma-algebra generated by $\{X_i\}_{i=\ell}^m$. The \textit{strong mixing coefficients} $\{\alpha(j)\}_{j\in\N}$ of $\{X_i\}_{i\in\N}$ are defined by 
	\begin{equation*}
		\alpha(j)=\sup_{k\in\N}\{\abs{P(A\cap B)-P(A)P(B)}:A\in\mathcal F_{1}^k, B\in\mathcal F_{k+j}^\infty\}, \quad j\in\N.
	\end{equation*}
	The sequence $\{X_i\}_{i\in\N}$ is said to be \textit{strongly mixing} (or $\alpha$-mixing) if $\lim_{j\to\infty} \alpha(j)=0$.
\end{df}

\begin{df}[Asymptotic orders]\label{d2}
	Let $\{X_i\}_{i\in\N}$  and $\{a_i\}_{i\in\N}$ be a sequence of random variables and a sequence of real numbers, respectively. 
	\begin{enumerate}
		\item[(a)]  $\{X_i\}_{i\in\N}$ is said to be of \textit{order almost completely smaller} than $\{a_i\}_{i\in\N^*}$ if, and only if,
		\begin{equation*}
			\forall \epsilon>0: \sum_{i\in\N} P(|X_i/a_i|>\epsilon)<\infty,
		\end{equation*}
		and we write $X_n=\oac(a_n)$. In particular, if $X_n=Z_n-Z=\oac(1)$, then we say that $\{Z_n\}_{n\in\N}$ \textit{converges almost completely} to the random variable $Z$. 
		
		\item[(b)] $\{X_i\}_{i\in\N}$ is said to be of \textit{order almost completely less than or equal to} that of $\{a_i\}_{i\in\N}$ if, and only if,
		\begin{equation*}
			\exists \epsilon>0: \sum_{i\in\N} P(|X_i/a_i|>\epsilon)<\infty,
		\end{equation*}
		and we write $X_n=\Oac (a_n)$.
		
		\item[(c)] We say that $a_n=\Theta(b_n)$, with $\{b_n\}_{n\in\N}$ being a sequence of real numbers, if there are $C_1,C_2>0$ such that $C_1 \leq \abs{a_n/b_n}\leq C_2$ for all $n$ sufficiently large.
	\end{enumerate} 
\end{df}

The asymptotic orders  defined above are consistent with the asymptotic notations commonly found in the literature\footnote{For comparison, see the Sections 1.4 and 2.1 of \cite{lehmann} .}. It can be seen that the almost complete convergence is a mode of strong convergence in the sense that $X_n=\oac(1)$ implies  $P(\limsup_{n\to\infty}\{\abs{X_n}>\epsilon\})=0$ for any $\epsilon>0$, by the Borel-Cantelli's lemma. In words, when $\{X_n\}_{n\in\N}$ converges almost completely, it also converges almost surely.

Now, we introduce some useful notations. 
Let $x\in\F$ be fixed and denote by $B(x,r)=\{x'\in\F: d(x,x')\leq r\}$ a closed ball of center $x$ and radius $r>0$ and by $P_{i,j}$  the pushforward measure induced by a random pair $(\rchi_i,\rchi_j)$. 

Let $[t]$ denote the set $\{1,\dotsc,t\}, \forall t\in\N$,  and $\R^*_+$ be the set of strictly positive real numbers. Define, for any  $m\in\N$ and $i\in [n]$, the operator $\rgamma_{m,i}:\F\to\R^*_+$ by $x\mapsto E(\abs{\varphi(Y_i)}^m\mid \rchi_i=x),$
and  define, $\forall i,j\in[n], \forall r_1,r_2,r_3,r_4>0$ and $\forall x'\in\F$,
\begin{align*}
	&\rphi_{x,i}(r_1,r_2)=P(r_1\leq d(x,\rchi_i)\leq r_2),\\
	&\Psi_{x,x',i,j}(r_1,r_2,r_3,r_4)=P(r_1\leq d(x,\rchi_i)\leq r_2,r_3\leq d(x',\rchi_j)\leq r_4).
\end{align*}
Whenever there is no risk of confusion, we use the notations $\rphi_{x,i}(r_1)\coloneqq \rphi_{x,i}(0,r_1)$,  $\rphi_{x}(r_1)\coloneqq\max_{s\in[n]}\rphi_{x,s}(r_1)$, $\Psi_{x,x,i,j}(r_1,r_2,r_3,r_4)\coloneqq \Psi_{x,i,j}(r_1,r_2,r_3,r_4)$ and $\Psi_{x,i,j}(r_1)\coloneqq \Psi_{x,i,j}(0,r_1,0,r_1)$.

In what follows, denote by $C$ and $c$, respectively, a generic large and a generic small positive constants that may take different values at different appearances.\footnote{Since the constants $0<C<\infty$, possibly distinct from each other, which appear in the text form a finite set, we are implicitly taking the greatest value among them. Likewise, we are implicitly taking the smallest value among the constants $0<c<\infty$.}

\subsection{Pointwise consistency}\label{sec3.2}

In this section, we provide convergence rates for the local linear estimator defined in \eqref{eq3}, pointwisely on $x\in\F$. The data is assumed to be strongly mixing with arithmetic mixing rates,  which is a standard choice in many regression frameworks \citep{hansen,leulmi,ferraty3}. It is worth noting that the data is allowed to be heterogeneously distributed. The asymptotic theory used to establish the almost sure convergence is based on the following set of assumptions:

\paragraph{Assumptions}
\vspace{1em}
The following assumptions are made throughout this section:
\vspace{.5em}
\begin{enumerate}
	\item[\textbf{A1}.] \label{a1} For all $h>0$ and $i\in[n]$, $ \rphi_{x,i}(h)>0$.
	\vspace{.5em}
	\item [\textbf{A2}.] There exist $0<b,C_2<\infty$ such that $\abs{m_\varphi(x_1)-m_\varphi(x_2)}\leq C_2 [d(x_1,x_2)]^b$ for every  $x_1,x_2\in B(x,h)$.
	\vspace{.5em}
	
	\item [\textbf{A3}.] There exist $0<c_3\leq C_3 <\infty$ such that $ c_3 d(x,x')\leq \abs{\beta(x,x')}\leq C_3 d(x,x')$ for all $x'\in\F$.
	
	\vspace{.5em}
	
	\item [\textbf{A4}.] For all $m\in\N$ and $i\in[n]$, the operator $\gamma_{m,i}$ is continuous at $x$. Moreover, there exist positive constants $c_4, C_4<\infty$ such that  $c_4<\min_{s\in[n]}\gamma_{1,s}(x)$, $\max_{s\in[n]}\gamma_{m,s}(x)<C_4, \forall m\geq 2,$ and $\sup_{i\neq j}  E(\abs{\varphi(Y_i)\varphi(Y_j)}\mid (\rchi_i,\rchi_j))\leq C_4$.
	
	\vspace{.5em}
	
	\item [\textbf{A5}.] The kernel function $K:\R\to\R_+$ is such that $\int_0^1 K(u)du=1$, its derivative $K'$ exists on $[0,1]$ and:
	\begin{enumerate}
		\item[(I)] $\exists 0<c_5\leq C_5< \infty: c_5 1_{[0,1]}\leq K \leq C_5 1_{[0,1]}$; or
		\item[(II)] $K(1)=0$, $\supp K=[0,1)$ and $-C'_5\leq K'\leq -c'_5$, for some $0<c'_5\leq C'_5$. 
	\end{enumerate}
	Whenever (II) holds, it is additionally required that $\exists c_0>0,\epsilon_0<1,n_0\in\N:\forall i\in[n]:\forall n>n_0: \int_0^{\epsilon_0}\rphi_{x,i}(uh)du>c_0 \rphi_{x,i}(h)$. 
	
	\vspace{.5em}
	
	\item [\textbf{A6}.] 
 \begin{enumerate}
 \item[(i)] There exist $c_6>0$ and $\epsilon^*<1$ such that
 $\rphi_{x,i}^{-1}(h)\int_0^{\epsilon^*} \rphi_{x,i}(zh,\epsilon h)\frac{d}{dz}(z^lK(z))dz>c_6$, for all $i,j\in[n], l\in\{2,4\}$ and $n$ sufficiently large ;
 \item[(ii)] $\forall i,j\in[n]:  h^{2}\iint_{B(x,h)^2}\beta(u,x)\beta(v,x)dP_{i,j}(u,v)=o\parent[\Big]{\iint_{B(x,h)^2}\beta(u,x)^2\beta(v,x)^2dP_{i,j}(u,v)}$.
  \end{enumerate}
	\vspace{.5em}
	
	\item [\textbf{A7}.] $\max_{s\in[n]}\rphi_{x,s}(h)=O(\min_{s\in[n]}\rphi_{x,s}(h))$.
	\vspace{.5em}
	
	\item [\textbf{A8}.]  The sequence $(Y_i,\rchi_i)_{i\in\N}$ is arithmetically strongly mixing  with rate $a>3$, i.e., $\exists  \delta, C_8:\forall n\in\N: \alpha(n)\leq C_8 n^{-(3+\delta)}$. Moreover, $\exists 0<\Delta<\min(a+1,\delta):\rphi_x(h)^{2(a+1)}\geq (\ln n)^{3(a+1)} n^{-\Delta}$;
    
	\vspace{.5em}
	\item[\textbf{A9}.] $\exists C_9,c_9>0:\forall i,j\in[n]: \exists 1/4<p_{1,i,j}\leq 1/2\leq p_{2,i,j}< 3/4$ such that 
	\begin{equation*}
		c_9[\rphi_{x,i}(h)\rphi_{x,j}(h)]^{1/2+p_{2,i,j}}\leq \Psi_{x,i,j}(h)\leq C_9 [\rphi_{x,i}(h)\rphi_{x,j}(h)]^{1/2+p_{1,i,j}}.
	\end{equation*}
	
	\vspace{.5em}
	
	\item[\textbf{A10}.] There is $c_{10}>0$ such that $\forall i,j\in[n]$ and $n$ large enough, if $K$ is of type (I) in \textbf{A5}, then
	\begin{equation*}
		\frac{1}{\Psi_{x,i,j}(h)}\int_0^{1} \Psi_{x,i,j}(zh,h,0,h)\frac{d}{dz}(z^2K(z))dz>c_{10},
	\end{equation*}
	and, additiionally, if $K$ is of type (II) in \textbf{A5}, it holds that
	\begin{equation*}
		-\frac{1}{\Psi_{x,i,j}(h)}\iint_0^{1} \Psi_{x,i,j}(zh,h,0,wh)\frac{d}{dz}(z^2K(z))K'(w)dzdw>c_{10}.
	\end{equation*}
\end{enumerate}
\vspace{0.5em}

Assumptions \textbf{A1}-\textbf{A4} are standard in the literature \citep[see][]{barrientos,ferraty,ferraty2,leulmi}. \textbf{A1} requires that the probability of observing each random variable $\rchi_i$ around $x$ is nonzero and \textbf{A2} assumes that $m_\varphi$ is $b$-H\"{o}lder continuous which will determine the bias order of our convergence problem as can be seen in Proposition \ref{p1}. In \textbf{A4}, the uniform bounds on $E(\abs{\varphi(Y_i)\varphi(Y_j)}\mid (X_i,X_j))$ and on $\gamma_{m,i}(x)$ provide a means to cope with the dependence of data. 

The set of kernel functions satisfying \textbf{A5} includes common choices such as the triangle, quadratic, cubic and uniform asymmetric kernels.\footnote{See the Definition 4.1, page 42, of \cite{ferraty}.} It is worth mentioning that our framework can be easily adapted to a more general support of form $\supp K=[0,L]$, for $ L>0$. For the sake of simplicity, we fixed $L=1$.  \textbf{A6} strengthens the assumptions (H6) and (H7) of \cite{barrientos}, originally made for independent data. \textbf{A6}(ii), which is included in assumption (H7) of \cite{leulmi},  specifies the local behavior of $\beta$ and \textbf{A6}(i) specifies the behavior of $h$ with respect to the small ball probabilities and the kernel function $K$. 

Since $\min_{s\in[n]}\rphi_{x,s}(h)\leq \rphi_{x,i}(h)\leq \max_{s\in[n]}\rphi_{x,s}(h)$ for all $i\in[n]$, the assumption that  $\max_{s\in[n]}\rphi_{x,s}(h)=O(\min_{s\in[n]}\rphi_{x,s}(h))$ in \textbf{A7}  implies that all  $\rphi_{x,i}(h), i\in[n],$ share the same asymptotic rate as $n\to\infty$. However, unlike the case of equally distributed data, we do not assume that $\rphi_{x,i}(h)=\rphi_{x,j}(h), i\neq j$.

The requirement that $\rphi_x(h)^{2(a+1)}\geq (\ln n)^{3(a+1)} n^{-\Delta}$ with $\Delta<a+1$, in \textbf{A8}, implies  that $\ln n/(n\rphi_{x}(h)^2)=o(1)$, and hence, that 
$\ln n/(n\rphi_{x}(h)^{4p_\text{max}-1})=o(1)$ if the number $p_\text{max}$ were less than $3/4$. This condition is crucial to ensure the consistency of $\hat m_{\varphi}$. It is a strengthening of the conventional assumption that $\ln n/(n\rphi_{x}(h))=o(1)$.\footnote{Since $(\ln n)^3\geq \ln n$ and $c\in (0,1)$, $\frac{\rphi_x(h)^2n}{\ln n}\geq n^{1-\Delta/(a+1)}\to\infty$ and so, $\frac{\ln n}{\rphi_x(h)^2n}\to0$ as $n\to\infty$. This, in turn, implies $\ln n/(n\rphi_{x}(h))\to 0$ as $\rphi_x(h)\geq \rphi_x(h)^2$ for $n$ large enough.}

In assuming that our process $(Y_i,\rchi_i)_{i\in\N}$ is strongly mixing, we are implicitly restricting the relation between the joint probability $\Psi_{x,i,j}(h)$ and the product of small ball probabilities $\rphi_{x,i}\rphi_{x,j}$ for long lag lengths (i.e., when $\abs{i-j}$ is relatively ``large'').\footnote{For more details, see Proposition 3 of the Supplementary Material.} This restriction is consistent with the definition of mixing, regarded as a notion of asymptotic independence. Proposition 3 of the Supplementary Material shows that if $1<n\min_{s\in[n]}\rphi_{x,s}(h)$, there will be indices $i,j\in[n]$ such that $\Psi_{x,i,j}(h)=\Theta(\rphi_{x,i}(h)\rphi_{x,j}(h))$.\footnote{Note that $1<n\min_{s\in[n]}\rphi_{x,s}(h)$ always holds if $\ln n/(n\min_{s\in[n]}\rphi_{x,s}(h))=o(1)$ and $n$ is sufficiently large.} For equally distributed and strong mixing data, \cite{leulmi} imposed that $C'\rphi_{x,1}(h)^{1+d}<\Psi_{x,i,j}(h)\leq C\rphi_{x,1}(h)^{1+d}$ for some $C',C>0$, some $d\in(0,1]$, and any $i,j\in[n]$.\footnote{Assumptions (H5a) and (H5b) of \cite{leulmi}, respectively} This situation,  however, is possible in their framework only if $d=1$, since $\Psi_{x,i,j}(h)$ cannot be $\Theta\parent[\big]{\rphi_{x,1}(h)^{1+d}}$ and $\Theta\parent[\big]{\rphi_{x,1}(h)^2}$ simultaneously, for $d\neq 1 $. In other words, the only possible uniform bounds for their setup would be  $C'\rphi_{x,1}(h)^{2}<\Psi_{x,i,j}(h)\leq C\rphi_{x,1}(h)^{2}$, or more generally, $\Psi_{x,i,j}(h)=\Theta(\rphi_{x,1}(h)^{2})$. 

In view of the above consideration, one can gain flexibility if instead \textbf{A9} were chosen. In this way, we are allowing $\Psi_{x,i,j}(h)$ to have distinct asymptotic orders along the pairs $(i,j)$ as $n\to\infty$.

\textbf{A10} is a technical assumption used for providing lower bounds for the expectation of local linear weights.\footnote{A similar assumption can be found in hypothesis (H7) of \cite{leulmi}.}  Similar to \textbf{A6}(i), \textbf{A10} specifies the local behavior of $h$ with respect to joint probabilities and the kernel function $K$. Proposition 4 in the Supplementary Material explores \textbf{A6}(i) and \textbf{A10} and shows that they hold for general processes of fractal order when the polynomial or uniform-type kernel functions are used.

We now state the almost complete convergence rate of $\hat m_\varphi(x)$. Let $p_{\max}=\max_{(i,j)\in[n]^2}p_{2,i,j}$ where $p_{2,i,j}$ is specified in \textbf{A9}.
\begin{thm}\label{teo1}
	Suppose that assumptions \textbf{A1}-\textbf{A10} are fullfiled. Then 
	\begin{equation}\label{eq4_}
		\hat m_\varphi (x)-m_\varphi(x)=O(h^b)+\Oac\parent[\bigg]{\sqrt{\frac{\ln n}{n\rphi_x(h)^{4p_\text{max}-1}}}}.
	\end{equation}
\end{thm}

Theorem \ref{teo1} shows that the heterogeneity and dependence of the data do not affect the deterministic, or the bias, part of the estimator $\hat m_{\varphi}(x)$ (for comparison, see Theorem 4.2 of \cite{barrientos}). Indeed, it only depends on the H\"older continuity order of the regression function $m_\varphi$. On the other hand, one can see that the convergence of the stochastic part can be slowered by the data dependence. Unlike the case of local constant estimator, here we have to deal with the joint probability $\Psi_{x,i,j}$ when providing  a lower bound for the expectation of the local linear weights. In its turn,  $\Psi_{x,i,j}$ is affected not only by the topological structure of $(\F,d)$, but also by its relation with $\rphi_{x,i}$ and $\rphi_{x,j}$ (i.e., by the dependence structure). The larger the exponent $p_{2,i,j}$ associated to the joint probability $\Psi_{x,i,j}$ according to  \textbf{A9}, the slower the convergence of the estimator. The reason is as follows: a large value of $p_{2,i,j}$ means that the joint probability of observing $\rchi_i$ and $\rchi_j$ rapidly decreases to zero as $n\to\infty$, which indicates that the data are overdispersed, leading to a less efficient convergence.

Since geometric mixing rates\footnote{We say that a random sequence $\{X_i\}_{i\in\N}$ is \textit{geometrically strongly mixing} if  its mixing coefficients satisfy $\alpha(k)\leq t^k, k\in\N,$ for some $t\in(0,1)$.} imply arithmetic mixing rates for any decay parameter $a>0$,\footnote{Indeed, if there exists $t\in(0,1)$ such that $\alpha(k)\leq Ct^k$, then  $\alpha(k)\leq Ck^{-a}$ for all $a>0$, since $t^k k^a\to 0$ as $k\to +\infty$.} Theorem \ref{teo1} also applies for geometrically $\alpha$-mixing data. 

It is known that the almost complete convergence implies  the convergence in probability. The next result provides convergence rates in probability under slightly weaker conditions. 
\begin{coro}\label{coro1}
	Under the conditions of Theorem 1, except A6(i), it holds that
	\begin{equation}
		\hat m_\varphi (x)-m_\varphi(x)=O(h^b)+O_p\parent[\bigg]{\sqrt{\frac{\ln n}{n\rphi_x(h)^{4p_\text{max}-1}}}}.
	\end{equation}
\end{coro}

In particular, if the data is independent (and thus, $\alpha$-mixing with mixing coefficient zero), then the estimator converges in the standard almost complete convergence rate (see Theorem 4.2 of \cite{barrientos} or Corollary 11.6 of \cite{ferraty}). This result is stated as follows.
\begin{coro}\label{coro2}
	Let the conditions of Theorem \ref{teo1} be satisfied. In addition, if $\{(\rchi_i,Y_i)\}_{i\in\N}$ is independent, it follows that
	\begin{equation}\label{e7}
		\hat m_\varphi (x)-m_\varphi(x)=O(h^b)+\Oac \parent[\Bigg]{\sqrt{\frac{\ln n}{n\rphi_{x}(h)}}\ }.
	\end{equation}
\end{coro}

\subsection{Uniform consistency}\label{sec3.3}

In this section, rates of almost sure convergence are established uniformly on a compact subset $S$ of the semimetric space $(\F,d)$. The main tool to cope with uniformity consists in covering $S$ with a finite number of balls. For this reason, the following topological concept introduced by \cite{kolmogorov} will be useful.

\begin{df}[Kolmogorov's entropy]
	Let $S$ be a subset of $(\F,d)$ and let $\epsilon>0$ be given. A finite set of elements $x_1,\dotsc,x_N\in\F$ is called an \textit{$\epsilon$-net} for $S$ if $S\subseteq \bigcup_{k=1}^N \{x\in\F:d(x,x_k)<\epsilon\}$. The quantity $\Phi_S(\epsilon)=\ln (N_\epsilon(S))$, where $N_\epsilon(S)$ is the minimum number of open balls in $\F$ of radius $\epsilon$ which is necessary to cover $S$, is called \textit{Kolmogorov's $\epsilon$-entropy} of the set $S$.
\end{df}

\vspace{-1em}
\paragraph{Assumptions}
Suppose that $\{x_1,\dotsc,x_{N_{r_n}(S)}\}\subseteq S$ is an $r_n$-net for $S$ with $\{r_n\}_{N\in\N}$ being a positive real sequence. Let $\overset{a}{\approx}$ denote asymptotic equivalence. The assumptions needed for the asymptotic results are listed as follows. 
\vspace{0.5em}
\begin{enumerate}
	\item[\textbf{H1}.] There exist a differentiable function $\rphi$ and constants $c,C>0$ such that $\forall x\in S, h>0, i\in[n]: 0<c\rphi(h)\leq \rphi_{x,i}(h)\leq C\rphi(h)$. Moreover, the function  $\rphi$ and its derivative $\rphi'$ are such that $\lim_{\eta\to 0}\rphi(\eta)=0$ and $\exists \eta_0>0:\forall \eta<\eta_0: \rphi'(\eta)<C$, respectively.
	\vspace{0.5em}
	\item [\textbf{H2}.] There exist $0<b,C<\infty$ such that for all $x_1\in S$ and all $x_2\in B(x_1,h)$ it holds that $\abs{m_\varphi(x_1)-m_\varphi(x_2)}\leq C [d(x_1,x_2)]^b$;
	\vspace{0.5em}
	
	\item [\textbf{H3}.] The function $\beta(\cdot,\cdot)$ satisfies \textbf{A3} uniformly on $x\in S$ and the Lipschitz condition that $\exists C>0:\forall x\in \F:\forall x_1,x_2\in S: \abs{\beta(x,x_1)-\beta(x,x_2)}\leq Cd(x_1,x_2)$;

	\vspace{0.5em}
	\item [\textbf{H4}.] The kernel function $K$ is Lipschitz continuous on $[0,1]$ and satisfies \textbf{A5}(I) or \textbf{A5}(II). If $K(1)=0$, the function $\rphi_{x,i}$ has to fulfill the additional condition that
	$ \exists c_0>0,\epsilon_0<1,n_0\in\N^*:\forall i\in[n]:\forall n>n_0: \inf_{x\in S}\int_0^{\epsilon_0}\rphi_{x,i}(uh)du>c_0 \inf_{x\in S}\rphi_{x,i}(h)$; 
	
	\vspace{0.5em}
	
	\item [\textbf{H5}.] The sequence $(Y_i,\rchi_i)_{i\in\N}$ is geometrically strongly mixing, i.e., $\alpha(k)\leq t^k, k\in\N,$ for some $t\in(0,1)$. Moreover, $\exists\Delta_1\in(0,1)$ such that $\rphi_{x}(h)^2\geq (\ln n)^3/n^{\Delta_1}$.
	
	\vspace{0.5em}
	
	\item[\textbf{H6}.] $\exists C,c>0:\forall x,x'\in S: \forall i,j\in[n]: \exists 1/4<p_{1,i,j}\leq 1/2\leq p_{2,i,j}< 3/4$:
	$
	c[\rphi_{x,i}(h)\rphi_{x',j}(h)]^{1/2+p_{2,i,j}}\leq \Psi_{x,x',i,j}(h)\leq C [\rphi_{x,i}(h)\rphi_{x',j}(h)]^{1/2+p_{1,i,j}}.
	$
	\vspace{.5em}
	\item [\textbf{H7}.] Uniformly on $x\in S$, the following assumptions hold: (i) \textbf{A4} for $m\geq 1$; (ii) \textbf{A6}; and (iii) \textbf{A10}.
	
	\vspace{.5em}
	
	\item [\textbf{H8}.] $r_n=O(\ln n/n)$ and $\Phi_S(\ln n/n)\overset{a}{\approx}C\ln n$.
\end{enumerate}
\vspace{0.5em}

The set of assumptions \textbf{H1}-\textbf{H8} is, roughly, an adaptation of the conditions \textbf{A1} - \textbf{A10} to the uniform case. \textbf{H1} is similar to assumptions (H1) and (H5a) of \cite{ferraty2} or (4) and (5) of \cite{benhenni}. \textbf{H3} is identical to assumptions (U3) of \cite{messaci} or \cite{leulmi}. Assumptions \textbf{H4} and \textbf{H8} are related to (H4) and Example 4 of \cite{ferraty}, respectively. The mixing decay in H5 has already been investigated  by several works \citep{truong,vogt,ferraty}, and, here, was useful to establish the asymptotic order of the stochastic part of the estimator (see the proof of Proposition \ref{p4}).

The following theorem states the uniform almost complete rate of convergence of the estimator defined in \eqref{eq3}.
\begin{thm}\label{teo2}
	Suppose that assumptions \textbf{H1}-\textbf{H8} are fullfiled. Then 
	\begin{equation*}
		\sup_{x\in S} \abs[\big]{\hat m_\varphi (x)-m_\varphi(x)}=O(h^b)+\Oac\parent[\bigg]{\sqrt{\frac{\ln n}{n\rphi_x(h)^{4p_\text{max}-1}}}}.
	\end{equation*}
\end{thm}
According to Theorem \ref{teo2}, we can obtain the same convergence rate as that of Theorem \ref{teo1}, uniformly on $S$. Moreover, one can check that the conclusions in Corollaries \ref{coro1} and \ref{coro2}  can be analogously obtained uniformly on $S$.

\section{Application to Wiener processes and a simulation study}\label{sec4}

Consider the space of square integrable real-valued functions on $[0,1]$, denoted as $\F=L^2[0,1]$, equipped with the standard inner product $<x_1,x_2> \coloneqq \int_0^1 x_1(t)x_2(t)dt$, $\forall x_1,x_2\in\F$. It is known that $\F$  with this inner product is a separable Hilbert space. Let $\{\rchi_i\}_{i=1}^n$ be a collection of $n$  independent standard Wiener processes on $[0,1]$ in $\F$ (also known as \textit{Brownian motion}).
Since each $\rchi_i\coloneqq\{\rchi_i(t), 0\leq t\leq 1\}$ is a second order zero-mean process with continuous covariance function $E(\rchi_i(t)\rchi_i(s))\coloneqq \min (t,s),  \forall t,s\in[0,1]$, we can expand $\rchi_i(t)$ through the Karhunen-Lo\`eve theorem as follows
\begin{equation}\label{eqregre}
    \rchi_i(t)=\sum_{j=1}^\infty v_j(t)N_{i,j},\quad t\in[0,1],
\end{equation}
where
$v_j(t)=\sqrt{2}\sin((j-1/2)\pi t), j\in\N$, are the eigenfunctions of the  Hilbert-Schmidt integral operator on $L^2[0,1]$ corresponding to the decreasingly ordered eigenvalues  $\lambda_j=[(j-1/2)\pi]^{-2}$, and $\{N_{i,j}\}_{j\in\N}$ is a sequence of independent Gaussian random variables such that  $N_{i,j}\sim N(0,\lambda_j)$.

 In this section, we compare the performance of the functional local linear regression operator (FLL) with that of the functional local constant (FLC) in a simulation study.

\emph{Data Generating Process:}
The explanatory curves are given by \eqref{eqregre} and are evaluated on a grid of 100 equally spaced points in $(0, 1)$. The dependent scalar variable $Y$, is defined as
\begin{equation*}
	Y_i = \sqrt{N_{i,1} + N_{i,2}} + \epsilon_i \text{ for } i = 1, \dots, n,
\end{equation*}
where the errors $\epsilon_i$ follow a stationary AR(1) process
\begin{equation} \label{eq7_}
	\epsilon_i = \alpha \epsilon_{i - 1} + u_i,
\end{equation}
where $u_i \sim N(0, 0.01)$ and $\alpha\in \{ 0, 1 / 3,2 / 3\}$. The experiment involves $n_{r} = 250$ Monte Carlo replicates.

\emph{Performance evaluation:}
In order to evaluate performance of both FLC and FLL estimators, we compute the mean squared prediction error (MSPE) for the estimator $s$ and replication $j$ as follows:

\begin{equation}\label{eq8_}
	MSPE_{s}^{[j]} = \frac{1}{n} \sum_{i = 1}^n \left(
	\hat Y_{i, s}^{[j]} - m \left(\rchi_{i}^{[j]} \right)
	\right)^2, \ j = 1, 2, \dots, n_{r}
\end{equation}
where 
$\hat Y_{i, s}^{[j]}$ is the prediction of $Y_{i}^{[j]}$ for the estimator $s \in \{ \text{FLC, FLL} \}$, and $m \left(\rchi_{i}^{[j]} \right) := \sqrt{N_{i,1}^{[j]} + N_{i,2}^{[j]}}$.  

\emph{Estimation details:}
The estimators FLL and FLC share the following general formula:
\begin{equation}\label{eq9_}
	\hat{m}(x) = \frac{\sum_{i,j=1}^n w_{i,j}(x) Y_j}{\sum_{i,j=1}^n w_{i,j}(x)},
\end{equation}
where for FLL, $w_{i,j}(x)$ is given as in the equation (\ref{eq3}) and for FLC, $w_{i,j}(x)$ simplifies for all $i=1,\dotsc, n$ to $w_{j}(x) = K(d(\rchi_j, x) / h)$.
For the kernel density we use the following polynomial kernel, which satisfy all the requirements of Theorems \ref{teo1} and \ref{teo2}.
\begin{equation}\label{eq10_}
	K(u) = \frac{3}{2} \left(1 - u^2 \right) I_{[0, 1]}(u).
\end{equation}

In order to select the bandwidth $h$, we use a leave-one-out cross-validation procedure that may be described as follows. Given a sample $(X_i, Y_i)$, $i = 1, 2, \dots, n$, the optimal bandwidth $h_{opt}$ is defined as
\begin{equation}\label{eq11_}
	h_{opt} = \arg\min_{h} n^{-1} \sum_{k = 1}^{n} \biggl(Y_k - \hat{m}_{(-k)}(X_k) \biggr)^2,
\end{equation}
where 
\begin{equation}\label{eq12_}
	\hat{m}_{(-k)}(X_k) = \frac{\sum_{i,j=1; i,j \ne k}^n w_{i,j}(X_k) Y_j}{\sum_{i,j=1; i,j \ne k}^n w_{i,j}(X_k)}.
\end{equation}

For the locating functions $\beta(\cdot, \cdot)$ and $d(\cdot, \cdot)$ we consider the PCA semimetric, which is defined in \cite{ferraty} and may be summarized as follows. Under the assumption that $E \int \mathbf{\rchi}^2(s)ds < \infty$, the~following expansion holds
\begin{equation} \label{eq13_}
	\rchi = \sum_{k = 1}^{\infty} \left( \int \rchi(s) v_k(s) ds \right) v_k,
\end{equation}  
\noindent where $v_1, v_2, \ldots$ are orthonormal eigenfunctions of the covariance operator 
\[
\Gamma_{\rchi}(t,s) = E \left(\rchi(t)\rchi(s) \right). 
\]    
From an empirical point of view, given an integer $r$, let
\begin{equation} \label{eq14_}
	\tilde{\rchi}^{(r)} = \sum_{k = 1}^{r} \left( \int \rchi(s) v_k(s) ds \right) v_k,
\end{equation}  
be a truncated version of $\rchi$. Based on the $L^2$-norm, for all \textcolor{purple}{$(X_1,X_2) \in \F^2$}, the following parametrized family of semi-metrics may be defined
\begin{equation}\label{eq15_}
	d_{r}^{PCA}(X_1,X_2) = \sqrt{\sum_{k=1}^r \left( \int(X_1(s) - X_2(s))v_{k}(s)ds \right)^2}.
\end{equation}

In order to estimate the bandwidth $h$ and the PCA semimetric parameter $r$, we apply the cross-validation procedure described in \eqref{eq11_} and \eqref{eq12_}. For some candidates $r$ and $h$ we choose the pair $(r_{opt}, h_{opt})$ that produces the smallest value in \eqref{eq11_}.
It is important to note that FLL requires two values of $r$, one associated with $\beta (\cdot,\cdot)$ and the other associated with $d(\cdot,\cdot)$.

\emph{Results:} For performance comparison, we report the distributions of the mean squared prediction errors (MSPE), as stated in equation \eqref{eq8_}, of the FLC and FLL estimators via the boxplots displayed in Figure \ref{fig:mspe}. Three pairs of boxplots are shown, each one related to a different value for the coefficient of the AR(1) process that characterizes the error sequence, as highlighted in equation \eqref{eq7_}.

In general, 
we see that the performance of both estimators slightly degrades as the level of dependence in the error sequence increases. Comparing one estimator with the other, FLL clearly outperforms FLC in terms of MSPE, having smaller median and smaller interquartile range compared to FLC. This improved performance of FLL over FLC is consistent across all levels of dependence in the error sequence considered. The better performance of the local linear functional estimator compared to the local constant is also documented in \cite{barrientos} and in \cite{leulmi}. 

\begin{figure}[ht]
	\caption{
		\textbf{MSPE - Simulation Study}\\[0.2cm]
		\scriptsize{
			The figure displays the boxplot of the mean squared predictive error (MSPE), as described in equation \eqref{eq8_}, for both FLC and FLL estimators. Three values for the coefficient of the AR(1) process that characterizes the error sequence are presented: $0, 1 / 3$ and $2 / 3$.
		}
	}
	\centering
	\includegraphics[width=2.5in, height=2.5in]{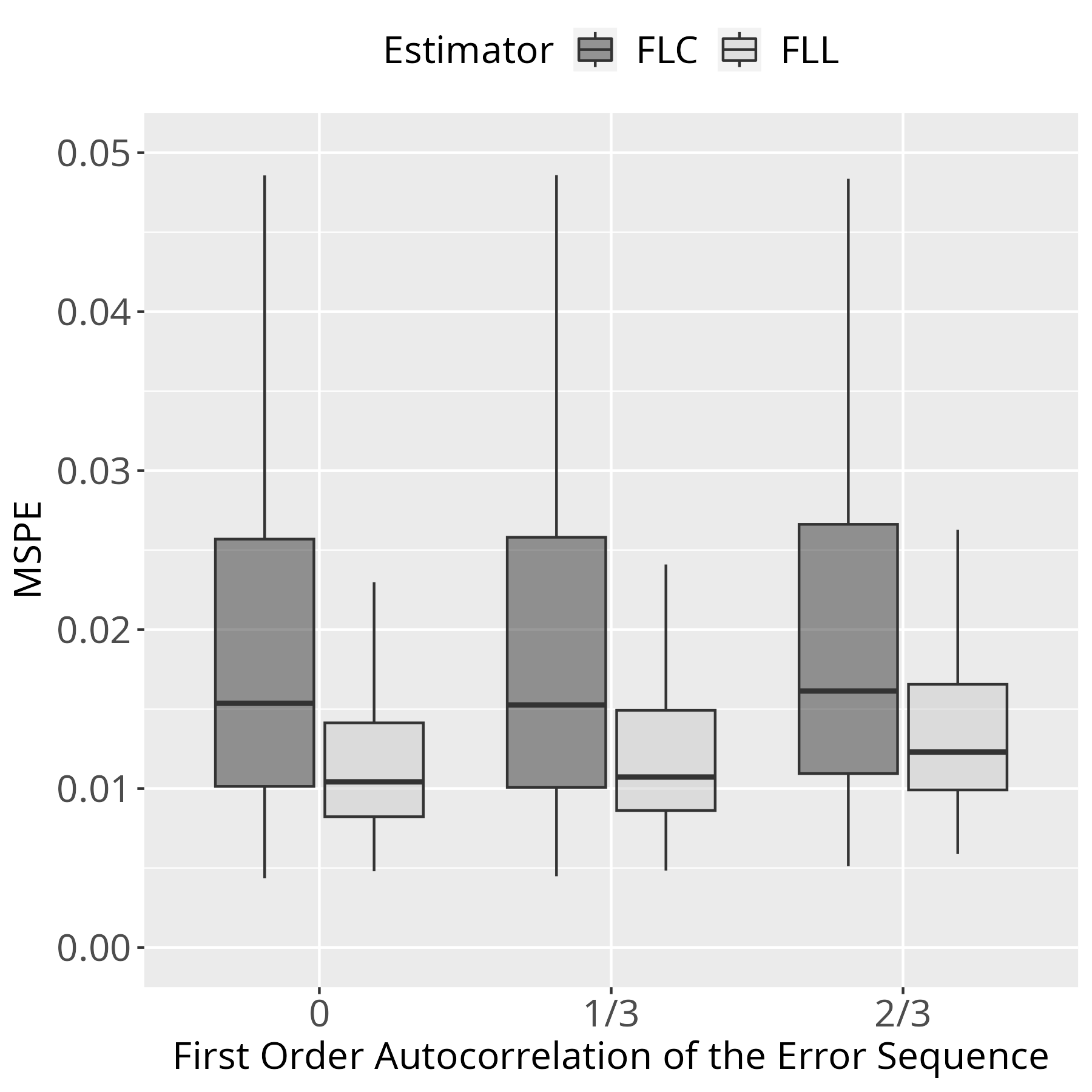}\\[-0.2cm]
	\label{fig:mspe}    
\end{figure}

\section{Real Data Application}\label{sec5}

In this section, we compare FLL and FLC estimators in a one step ahead energy consumption forecast situation.

The empirical data set we use here is hourly energy consumption data from America Electric Power (AEP). This data is under public domain license and it is available on the Kaggle website (https://www.kaggle.com/robikscube/hourly-energy-consumption). 

We consider the link between the logarithm (log) of hourly energy consumption of a day (the explanatory variable) and the log of total consumption of the following day (the response variable). Therefore, the explanatory variable is a curve discretized over 24 points and the response is a scalar variable.

The data ranges from 2004-10-01 to 2018-08-02, giving $T = 5054$ days of observations. A rolling window scheme is considered with window length equal to $W = 1081$ (3 comercial years plus the forecast horizon). Therefore, we generate $T_{out} = 3973$ one step ahead forecasts of the daily energy consumption.

\begin{figure}[ht]
	\caption{
		\textbf{Time Series of Energy Consumption}\\[0.2cm]
		\scriptsize{
			The figure displays the sample of hourly energy consumption data from America Electric Power (AEP). The data ranges from 2004-10-01 to 2018-08-02, giving $T = 5054$ days of observations. A rolling window scheme is considered with window length equal to $W = 1081$. Thus, we generate $T_{out} = 3973$ one step ahead forecasts of the daily energy consumption.
		}
	}
	\centering
	\includegraphics[width=4.4in, height=2.2in]{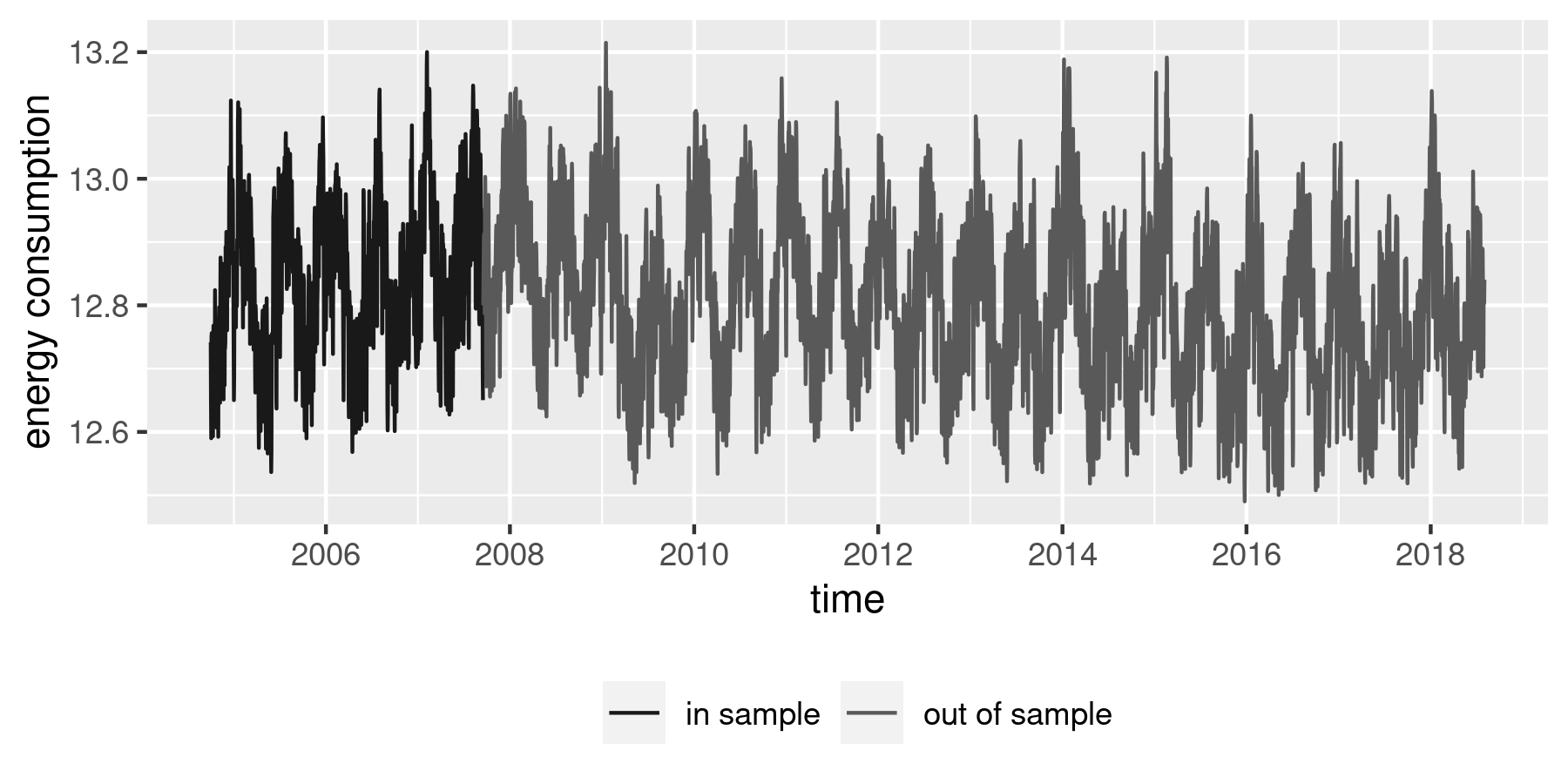}\\[-0.2cm]
	\label{fig:descriptive}    
\end{figure}

\emph{Performance evaluation:}
We graphically analyze the cumulative squared forecast error (CSFE) as proposed by \cite{WelchGoyal2008}. The CSFE specific to our case may be defined as

\begin{equation} \label{eq16_}
	CSFE_{i_{t}} = \sum_{j = i_{1}}^{i_{t}} \left[
	\left(\hat y_{j + 1 | j}^{FLC} - y_{j + 1} \right)^2 -
	\left(\hat y_{j + 1 | j}^{FLL} - y_{j + 1} \right)^2
	\right],
\end{equation}
where $\hat y_{j + 1 | j}^{FLC}$ and $\hat y_{j + 1 | j}^{FLL}$ are the one step ahead forecast of FLC and FLL estimators, respectively; $y_{j+1}$ is the observed response variable at time $j + 1$ and $i_{t}$, with $t = 1, 2, \dots, T_{out}$ are the indexes of the observations that are relevant to the forecast exercise.
Increasing CSFE implies better predictive performance of the FLL estimator compared to the FLC estimator, while decreasing CSFE implies otherwise.
In order to test and compare the predictive ability of FLL and FLC we apply the test of conditional predictive ability proposed by \cite{GiacominiWhite2006}, \textcolor{purple}{shortly referred to as \textit{GW-test}}. The null hypothesis here is that FLC performs at least as good \textcolor{purple}{as} FLL in terms of squared forecasting errors.

\emph{Estimation details:}
Now, we detail the estimation procedure for FLL and FLC. In order to select the bandwidth $h$, we use the leave-one-out cross-validation procedure described in equations (\ref{eq11_}) and (\ref{eq12_}). For the locating functions $\beta(\cdot, \cdot)$ and $d(\cdot, \cdot)$, we choose the PCA semimetric, which is suitable when the number of discretized points is small. 
In order to estimate the bandwidth $h$ and the PCA semimetric parameter $r$, we apply the same scheme \textcolor{purple}{as} described in the simulation study.

\begin{figure}[h]
	\caption{
		\textbf{CSFE - Energy Forecast Application}\\[0.2cm]
		\scriptsize{
			The figure displays the cumulative squared forecast error (CSFE) as described in equation \eqref{eq16_}. Increasing CSFE implies better predictive performance of the FLL estimator compared to the FLC estimator, while decreasing CSFE implies otherwise.
		}
	}
	\centering
	\includegraphics[width=2.5in, height=2.5in]{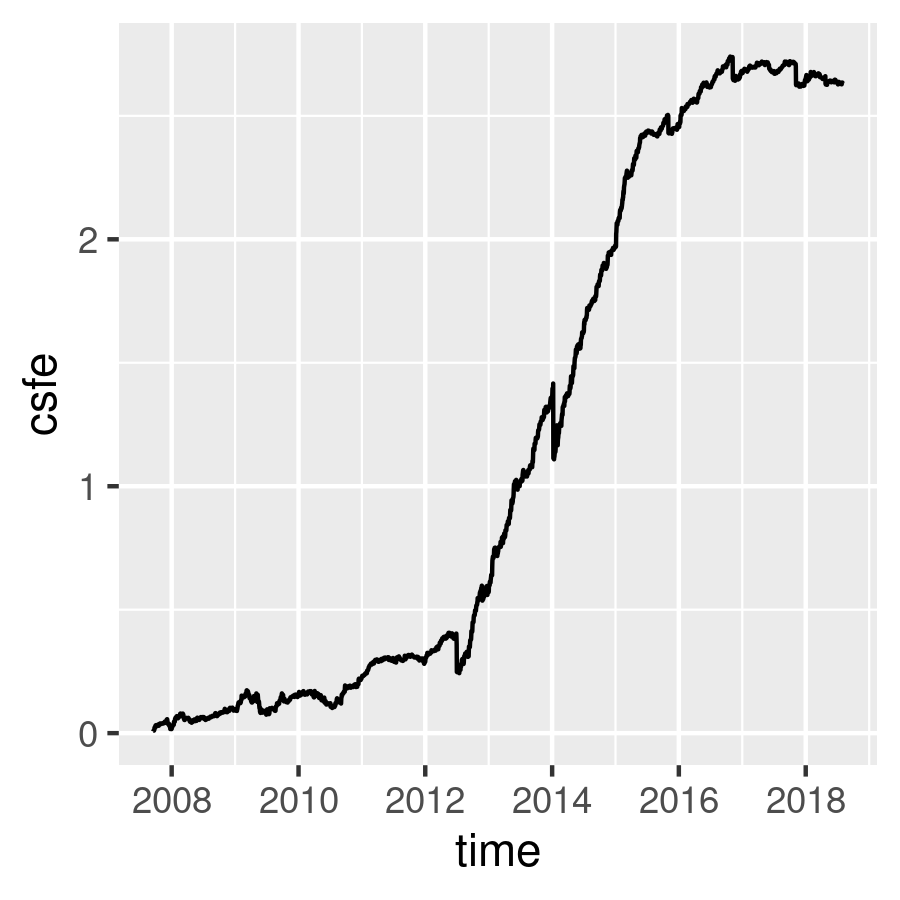}\\[-0.2cm]
	\label{fig:csfe}    
\end{figure}

\emph{Results:}
 The results for the CSFE are presented in Figure \ref{fig:csfe}. The overall conclusion is that FLL tends to outperform FLC during almost the entire period considered. One exception is the last part of the sample, starting approximately from the first quarter of 2017.     
Using squared forecasting errors as performance criteria, the GW-test rejects the null with p-value equal to $1.17 \times 10^{-08}$, which means that the forecasts of the FLL estimators are significantly more accurate than those of the FLC.

\section{Conclusion}\label{sec6}
The main contribution of this paper is a step towards the functional nonparametric modeling when the data is heterogeneously distributed and strongly mixing. Our theoretical results show that the almost complete convergence rate can be slower in the presence of data dependence. This is so because  our framework links the joint concentration properties of the data with its dependence. When the data is independent, however, the standard rate of convergence is obtained. Moreover, under our conditions, it is demonstrated that the pointwise and uniform convergence rates are the same on compact sets. The simulation results showed a good overall performance of the functional local linear estimator in comparison with the local constant estimator. In addition, a one step ahead energy consumption forecasting exercise illustrates that the forecasts of the former estimator are significantly more accurate than those of the latter.

\section*{Declarations}
\textbf{Conflict of interest}.The authors have no competing interests to declare that are relevant to the content of this article.

\bibliographystyle{apalike}
\bibliography{biblio}
\section*{Appendix A: Auxiliary results}
\paragraph{Lemmas}
\vspace{1em}

In this section, whenever possible, we will omit the dependence of the following terms on $x$: $K_i(x)=K_i$ and $\beta_i(x)=\beta_i$. In addition, define $\N_0=\N\cup\{0\}$. The proofs of the lemmas below can be found in Section 2 of the Supplementary Material.
\begin{lemma}\label{l1}
	Let the assumptions \textup{\textbf{A1}, \textbf{A3}, \textbf{A4}, \textbf{A5}, \textbf{A6}(i) } and   \textup{\textbf{A10}} hold. Then for all $i,j\in[n]$ and $n$ sufficiently large we have that:
	\begin{enumerate}
		\item[(i)] $ E(K_i^q \abs{\beta_i}^\ell)\leq Ch^\ell\rphi_{x,i}(h), \forall (q,\ell)\in \N\times \N_0$;
		\item[(ii)] $E(K_i K_j \abs{\beta_i}^{\ell_1}\abs{\beta_j}^{\ell_2})\leq C h^{\ell_1+\ell_2} \Psi_{x,i,j}(h), \forall \ell_1,\ell_2\in[2]$;
		\item[(iii)] $E(K_i K_j \beta_i^2)>c^*h^2\Psi_{x,i,j}(h)$ where $c^*=c_3c_5\min(1,c_{10})>0$;
		\item[(iv)] $E(K_i^2 \beta_i^\ell)> ch^\ell\rphi_{x,i}(h), \forall \ell\in\{0,2,4\}$;
        \item[(v)] $I(0\leq d(x,\rchi_i)\leq h)E(\abs{\varphi_i}^\ell|\rchi_i)\leq C, \forall \ell\in\N_0$.
	\end{enumerate}
\end{lemma}

\begin{lemma}\label{l2}
	The cardinal of the set $S_1=\{(i,j)\in[n]^2: 1\leq\abs{i-j}\leq a_n\}$ is asymptotically equivalent to $2na_n$, where $a_n$ is some positive sequence diverging to infinity.
\end{lemma}

Consider the following sum of covariances
\begin{equation*}
	S_{n,\ell,k}^2(x)\coloneqq \sum_{i,j=1}^n \abs[\big]{\cov \parent[\big]{\Lambda_{i}^{(k,\ell)}(x),\Lambda_{j}^{(k,\ell)}(x)}}, 
\end{equation*}
where
\begin{equation*}
	\Lambda_i^{(k,\ell)}(x)\coloneqq \frac{1}{h^k}\chave{K_i(x)\beta_i(x)^k\varphi_i^\ell-E\colc{K_i(x)\beta_i(x)^k\varphi_i^\ell}},
\end{equation*}
for $i\in[n]$ and $\ell,k\in\N_0$.

\begin{lemma}\label{l3}
	Let the assumptions  \textup{\textbf{A1}-\textbf{A5}, \textbf{A8}} and  \textup{\textbf{A9}} be fulfilled. Then for all $k\in\{0,1,2\}$, $\ell\in\{0,1\}$, it follows that
	\begin{equation}\label{eq17_}
		S_{n,\ell,k}^2(x)=O \parent[\big]{n\rphi_{x}(h)}.
	\end{equation}
	If in addition \textbf{A6}(i) and \textbf{A7} hold, then 
	\begin{equation}\label{eq18_}
		S_{n,\ell,k}^2(x)=\Theta \parent[\big]{n\rphi_{x}(h)}.
	\end{equation}
\end{lemma}

\begin{lemma}\label{l4}
	Let the assumptions  \textup{\textbf{H1}, \textbf{H3}, \textbf{H4}, \textbf{H7}(ii)} and  \textup{\textbf{H7}(iii)} hold. Then for any $i,j\in[n]$ and $n$ sufficiently large we have that:
	\begin{enumerate}
		\item[(i)] $ \sup_{x\in S}E(K_i^q \abs{\beta_i}^\ell)\leq Ch^\ell\sup_{x\in S}\rphi_{x,i}(h),\forall (q,\ell)\in \N\times \N_0$;
		\item[(ii)] $\sup_{x\in S}E(K_i K_j \abs{\beta_i}^{\ell_1}\abs{\beta_j}^{\ell_2})\leq C h^{\ell_1+\ell_2} \sup_{x\in S}\Psi_{x,i,j}(h)$,  for every $\ell_1,\ell_2\in[2]$;
		\item[(iii)] $\inf_{x\in S}E(K_i K_j \beta_i^2)>ch^2\inf_{x\in S}\Psi_{x,i,j}(h)$;
		\item[(iv)] $\inf_{x\in S}E(K_i^2 \beta_i^\ell)> ch^\ell\inf_{x\in S}\rphi_{x,i}(h)$, for all $\ell\in\{0,2,4\}$;
         \item[(v)] $I(0\leq d(x,\rchi_i)\leq h)E(\abs{\varphi_i}^\ell|\rchi_i)\leq C, \forall \ell\in\N_0, \forall x\in S$.
	\end{enumerate}
\end{lemma}

\begin{lemma}\label{l5}
	Suppose that assumptions  \textup{\textbf{H1}-\textbf{H6}} and  \textup{\textbf{H7}(i)-(ii)} are fulfilled. Then for all $k\in\{0,1,2\}$, $\ell\in\{0,1\}$, it follows that
	\begin{equation*}
		cn\rphi(h)\leq\inf_{x\in S}S_{n,\ell,k}^2(x)\leq \sup_{x\in S}S_{n,\ell,k}^2(x)\leq Cn\rphi(h).
	\end{equation*}
\end{lemma}

Now, define for all $i\in[n], \ \ell\in\{0,1\}$ and $x\in \F$, the random variable
$$T_i^\ell(x)=r_n \frac{\abs{\varphi_i}^\ell}{h}1_{B(x,h)\cup B(x_{j(x)},h)}(\rchi_i),$$
with $j(x)=\argmin_{j\in[N_{r_n}(S)]}d(x,x_{j})$. Moreover, let $M_i^\ell(x)=T_i^\ell(x)-ET_i^\ell(x)$ and $W_{n,\ell}^2(x)=\sum_{i,j=1}^n \abs{\cov (M_i^\ell(x),M_j^\ell(x))}$.
\begin{lemma}\label{l6}
	Suppose that the assumptions  \textup{\textbf{H1}, \textbf{H5}, \textbf{H6}} and  \textup{\textbf{H7}(i)} hold. Then for $l\in\{0,1\}$, it follows that
	\begin{equation*}
		c\frac{r_n^2}{h^2}n\rphi(h)\leq\inf_{x\in S} W^2_{n,\ell}(x)\leq \sup_{x\in S} W^2_{n,\ell}(x)\leq C\frac{r_n^2}{h^2}n\rphi(h),
	\end{equation*}
	and 
	\begin{equation*}
		\sup_{x\in S} \max_{i\in[n]}ET^\ell_i(x)=O \parent[\bigg]{\frac{r_n}{h}\rphi(h)}.
	\end{equation*}
\end{lemma}

\paragraph{Propositions}
\vspace{1em}
Let $m_\ell(x)=(1/\Gamma(x))\sum_{i\neq j}^n w_{i,j}(x) \varphi_j^\ell$ where $\Gamma(x)=\sum_{i\neq j}^n E(w_{i,j}(x))$, for $\ell\in\N$. Moreover, denote  $p_{\min}=\min_{(i,j)\in[n]^2}p_{1,i,j}$ and $p_{\max}=\max_{(i,j)\in[n]^2}p_{2,i,j}$.

\begin{prop}\label{p1}
	Suppose that assumptions  \textup{\textbf{A1}-\textbf{A3}, \textbf{A5}, \textbf{A6}(ii), \textbf{A7}, \textbf{A9}} and  \textup{\textbf{A10}} hold. Then
	\begin{equation*}
		m_\varphi (x)-E(m_1(x))=O(h^b).
	\end{equation*}	
\end{prop}

\noindent\textbf{Proof of Proposition \ref{p1}}
Assumption \textbf{A6}(ii) implies that 
\begin{equation*}
h^2E(K_jK_i\beta_i\beta_j)=o\parent[\bigg]{\iint_{B(x,h)^2}\beta(u,x)^2\beta(v,x)^2dP_{i,j}(u,v)}=o\parent[\big]{h^4\Psi_{x,i,j}(h)}.
\end{equation*}
 Then, using Lemma 1(iii) and \textbf{A9},
	\begin{align}\notag
		 E(w_{i,j}(x))&= E(K_jK_i\beta_i^2)- E(K_jK_i\beta_i\beta_j)\\\notag
		&> c^* h^2 \Psi_{x,i,j}(h)-c h^2\Psi_{x,i,j}(h)=h^2 \Psi_{x,i,j}(h)(c^*-c)\\\label{eq20_}
		&>ch^2 \Psi_{x,i,j}(h)>0,
	\end{align}
	for some $c^*>0$, all $n$ sufficiently large and $c>0$ chosen small enough. 
 
By hypothesis, $\rchi_i$ is independent of $\epsilon_j$ and $E(\epsilon_j)=0$, $\forall i,j\in[n]$. Moreover, the regression function  $m_\varphi(\rchi_j)$ is $\sigma(\rchi_j,\rchi_i)$-measurable since $\sigma(\rchi_j)\subseteq \sigma(\rchi_j,\rchi_i), \forall i,j$. Then
 \begin{equation}\label{eq_20_}
     E(\varphi_j|(\rchi_i,\rchi_j))=E(m_\varphi(\rchi_j)+\epsilon_j|(\rchi_i,\rchi_j))=m_\varphi(\rchi_j)+E(\epsilon_j|(\rchi_i,\rchi_j))=m_\varphi(\rchi_j), \quad \forall i,j\in[n].
 \end{equation}
 Given a random variable $\rchi$, define its positive and negative parts by $\rchi^+\coloneqq \max(\rchi,0)$ and $\rchi^-\coloneqq \max(-\rchi,0)$, respectively. 
 Then, for all $n$ sufficiently large, we use the Law of Iterated Expectations, \eqref{eq20_}, \eqref{eq_20_} and \textbf{A2}  to obtain that
	\begin{align}\notag
		m_\varphi(x)-E(m_1(x))&=\frac{\Gamma(x)}{\Gamma(x)}m_\varphi(x)-\frac{1}{\Gamma(x)}\sum_{i\neq j}E\parent[\big]{w_{i,j}(x)\varphi_j}\\\notag
        &=\frac{1}{\Gamma(x)}\sum_{i\neq j}E\parent[\big]{w_{i,j}(x)m_\varphi(x)}-\frac{1}{\Gamma(x)}\sum_{i\neq j}E\parent[\big]{w_{i,j}(x)E\parent[\big]{\varphi_j|(\rchi_i,\rchi_j)}}
        \\\notag
		&=\frac{1}{\Gamma(x)}\sum_{i\neq j}E\colc[\big]{w_{i,j}(x)\parent[\big]{m_\varphi(x)-m_{\varphi}(\rchi_j)}}\\\notag
		&\leq \frac{1}{\Gamma(x)}\sum_{i\neq j} E\abs{w_{i,j}(x)}  \sup_{x'\in B(x,h)}\abs{m_\varphi(x)-m_{\varphi}(x')}\\\notag
		&\leq Ch^b \frac{\sum_{i\neq j}E\abs{w_{i,j}(x)}}{\sum_{i\neq j}E\parent{w_{i,j}(x)}}\\\notag
		&=Ch^b \frac{\sum_{i\neq j}E\parent[\big]{\abs{w_{i,j}(x)}-w_{i,j}(x)+w_{i,j}(x)}}{\sum_{i\neq j}E\parent{w_{i,j}(x)}}\\\notag 
		&=Ch^b \chave[\bigg]{\frac{\sum_{i\neq j}E\parent[\big]{\abs{w_{i,j}(x)}-w_{i,j}(x)}}{\sum_{i\neq j}E\parent{w_{i,j}(x)}}+1}\\\notag
    &=Ch^b \chave[\bigg]{\frac{\sum_{i\neq j}E\parent[\big]{w_{i,j}(x)^+ + w_{i,j}(x)^- -w_{i,j}(x)^+ +w_{i,j}(x)^-}}{\sum_{i\neq j}E\parent{w_{i,j}(x)}}+1}\\\label{eq21_}
    &\leq Ch^b \frac{\sum_{i\neq j}2E\parent[\big]{w_{i,j}(x)^-}}{\sum_{i\neq j}E\parent{w_{i,j}(x)}}.
	\end{align}
	Observe  that 
 \begin{align*}
     \{\omega\in\Omega: (K_iK_j(\beta_i-\beta_j)\beta_i)(\omega)<0\}
     &\subseteq \{\omega\in\Omega: \beta_j(\omega)>\beta_i(\omega)>0\}\\
     &\qquad \cup \{\omega\in\Omega: \beta_j(\omega)<\beta_i(\omega)<0\}.
 \end{align*}
 Then, using Lemma 1(ii),
	\begin{align}\notag
		E\parent[\big]{w_{i,j}^-}&=E\parent[\big]{-w_{i,j}I(w_{i,j}<0)}= E((K_i K_j \beta_j\beta_i-K_i K_j \beta_i^2)I(w_{i,j}<0))\\\notag
		&\leq  E\chave[\big]{(K_i K_j \beta_j\beta_i)I(w_{i,j}<0)(I(\beta_j>\beta_i>0)+I(\beta_j<\beta_i<0))}\\\notag
        &=\abs{E\parent[\big]{K_iK_j\beta_i\beta_jI(w_{i,j}<0)}}\\\label{eq_22_}
        &\leq \abs{E\parent[\big]{K_iK_j\beta_i\beta_j}}\leq C h^2\Psi_{x,i,j}(h).
	\end{align}
	Combining \eqref{eq20_} and \eqref{eq_22_}, we obtain that
	\begin{equation}\label{eq23_}
		\frac{\sum_{i\neq j}E\parent[\big]{w_{i,j}^-}}{\sum_{i\neq j}E\parent[\big]{w_{i,j}}}\leq C	\frac{\sum_{i\neq j}\Psi_{x,i,j}(h)}{\sum_{i\neq j}\Psi_{x,i,j}(h)}=C,
	\end{equation}
	for all $n$ sufficiently large. It is immediate from \eqref{eq21_} and \eqref{eq23_} that $	m_\varphi(x)-E(m_1(x))=O(h^b)$. 
    
    \qed

\begin{remark}
	Note that Lemmas 1 and 4 of \cite{leulmi} are proved using the same arguments as \cite{barrientos}. However, the proof of the latter authors is based on the equality $E(w_{i,j})=E(w_{1,2}), \forall i,j\in[n],$ which holds for independent and identically distributed data but is not at all obvious for dependent and identically distributed data.
\end{remark}

\begin{prop}\label{p2}
	Let the conditions \textbf{H1}-\textbf{H4}, \textbf{H6}, \textbf{H7}(ii) and \textbf{H7}(iii) hold. Then
	\begin{equation*}
		\sup_{x\in S}\abs{m_\varphi (x)-E(m_1(x))}=O(h^b).
	\end{equation*}	
\end{prop}
The proof follows along the same lines as the proof of Proposition 1, and thus omitted.

\begin{prop}\label{p3}
	If the assumptions \textbf{A1}-\textbf{A10} hold, then
	\begin{equation}\label{eq24_}
		m_1 (x)-Em_1(x)=\Oac\parent[\bigg]{\sqrt{\frac{\ln n}{n\rphi_x(h)^{4p_\text{max}-1}}}},
	\end{equation}
	and 
	\begin{equation}\label{eq25_}
		m_0 (x)-1=\Oac\parent[\bigg]{\sqrt{\frac{\ln n}{n\rphi_x(h)^{4p_\text{max}-1}}}}.
	\end{equation}
	If \textbf{A6}(i) is excluded, then
	\begin{equation}\label{eq26_}
		m_1 (x)-Em_1(x)=O_p\parent[\bigg]{\sqrt{\frac{\ln n}{n\rphi_x(h)^{4p_\text{max}-1}}}},
	\end{equation}
	and 
	\begin{equation}\label{eq27_}
		m_0 (x)-1=O_p\parent[\bigg]{\sqrt{\frac{\ln n}{n\rphi_x(h)^{4p_\text{max}-1}}}}.
	\end{equation}
\end{prop}
\noindent\textbf{Proof of Proposition \ref{p3}}
	We first prove \eqref{eq24_}. Following \cite{barrientos}, set 
	\begin{align*}
		m_1(x)&=\frac{1}{\Gamma(x)}\sum_{i,j=1}^n K_iK_j\beta_i^2\varphi_j-K_iK_j\beta_i\beta_j\varphi_j\\
		&=\colc[\bigg]{\frac{\parent[\big]{nh\rphi_{x}(h)}^2}{\Gamma(x)}}\sum_{i,j=1}^n\chave[\bigg]{ \colc[\bigg]{\frac{K_j\varphi_j}{n\rphi_{x}(h)}}\colc[\bigg]{\frac{K_i\beta_i^2}{nh^2\rphi_{x}(h)}}-\colc[\bigg]{\frac{K_i\beta_i}{nh\rphi_{x}(h)}}\colc[\bigg]{\frac{K_j\beta_j\varphi_j}{nh\rphi_{x}(h)}}}\\
		&\coloneqq Q[S_{2,1}S_{4,0}-S_{3,1}S_{3,0}]
	\end{align*}
	where, for $q\in\{2,3,4\}$ and $\ell\in\{0,1\}$,
	\begin{equation*}
		S_{q,\ell}\coloneqq\frac{1}{n\rphi_{x}(h)}\sum_{i=1}^n\frac{K_i\beta_i^{q-2}\varphi_i^\ell}{h^{q-2}} \text{ and } Q\coloneqq \frac{\parent[\big]{nh\rphi_{x}(h)}^2}{\Gamma(x)}.
	\end{equation*} 
	
	Hence, $$m_1(x)-Em_1(x)= Q\chave{[S_{2,1}S_{4,0}-E(S_{2,1}S_{4,0})]-[S_{3,1}S_{3,0}-E(S_{3,1}S_{3,0})]}.$$
	The first term in brackets above can be written as
	\begin{align*}
		S_{2,1}S_{4,0}-E(S_{2,1}S_{4,0})&=(S_{2,1}-ES_{2,1})(S_{4,0}-ES_{4,0})
		+ES_{4,0}(S_{2,1}-ES_{2,1})\\
		&+ES_{2,1}(S_{4,0}-ES_{4,0})-\cov(S_{2,1},S_{4,0}).
	\end{align*}
	The second term can also be represented analogously. 
	Therefore, the asymptotic order of $m_1(x)-Em_1(x)$ will be determined as soon as the following results are proven for $q\in\{2,3,4\}$ and $\ell\in\{0,1\}$:
	\begin{enumerate}
		\item[(a)] $ES_{q,\ell}=O(1)$;
		\item[(b)] $Q=O(\rphi_{x}(h)^{1-2p_\text{max}})$;
		\item[(c)] $S_{q,\ell}-ES_{q,\ell}=\Oac \parent[\Big]{\sqrt{\ln n/(n\rphi_{x}(h))}\ }$;
		\item[(d)] $\cov(S_{2,1},S_{4,0})=o\parent[\Big]{\sqrt{\ln n/(n\rphi_{x}(h))}\ }, \cov(S_{3,1},S_{3,0})=o\parent[\Big]{\sqrt{\ln n/(n\rphi_{x}(h))}\ }.$
	\end{enumerate}
	
	The first result, (a), can be obtained through Lemma 1(i)(v) as follows
	\begin{equation*}
		ES_{q,\ell}\leq \frac{C}{n\rphi_{x}(h)}\sum_{i=1}^n \frac{E(K_i\beta_i^{q-2})}{h^{q-2}}\leq \frac{C}{n\rphi_{x}(h)}\sum_{i=1}^n \frac{h^{q-2}\rphi_{x}(h)}{h^{q-2}}\leq C.
	\end{equation*}
	
	By \eqref{eq20_} in the proof of Proposition \ref{p1}, together with \textbf{A7} and \textbf{A9}, it can be seen that $\Gamma(x)\geq cn(n-1)h^2\rphi_{x}(h)^{1+2p_{\text{max}}}$ for $n$ sufficiently large. Then, (b) follows from
	\begin{equation*}
		Q= \frac{\parent[\big]{nh\rphi_{x}(h)}^2}{\Gamma(x)}\leq C\rphi_{x}(h)^{1-2p_\text{max}},
	\end{equation*}
	for all $n$ large enough.
	
	Next, (c) is proved by applying the Fuk-Nagaev's inequality \citep[Theorem 6.2 of][]{rio}. Write $S_{q,\ell}-ES_{q,\ell}=\colc[\big]{n\rphi_{x}(h)}^{-1}\sum_{i=1}^{n} Z_i^{(q,\ell)}$ where
	\begin{equation*}
		Z_i^{(q,\ell)}\coloneqq\Lambda_i^{(q-2,\ell)}(x)= \frac{1}{h^{q-2}}\chave{K_i\beta_i^{q-2}\varphi_i^\ell-E\colc{K_i\beta_i^{q-2}\varphi_i^\ell}}.
	\end{equation*}
	Obviously, $EZ_i^{(q,\ell)}=0$. Moreover, it can be shown that $E\abs[\big]{\Lambda_i^{(q-2,\ell)}}^r=O(\rphi_{x}(h))$, $\forall r\geq 2,$ and so $Z_i^{(q,\ell)}$ has finite variance (see inequality (4) in the proof of Lemma \ref{l3} in the Supplementary Material). For all $n$ large enough, Markov's inequality implies that
	\begin{equation*}
		P(\abs{Z_i^{(q,\ell)}}>t)\leq \frac{E\abs{Z_i^{(q,\ell)}}^r}{t^r}\leq \frac{C}{t^r}\rphi_{x}(h)\leq \frac{1}{t^r}. 
	\end{equation*}
	With these observations, the conditions of the Fuk-Nagaev's inequality are fulfilled. Thus we have that
	\begin{align}\notag
		P\parent[\bigg]{\abs[\bigg]{\sum_{i=1}^{n}Z_i^{(q,\ell)}}>4\lambda n\rphi_{x}(h)}&\leq 4\parent[\bigg]{1+\frac{(\lambda n\rphi_{x}(h))^2}{vS_{n,\ell,q-2}^2(x)}}^{-v/2}+4C\frac{n}{v}\parent[\bigg]{\frac{v}{\lambda n\rphi_{x}(h)}}^{r(a+1)/(a+r)}\\\label{eq28_}
		&\coloneqq M_{1,n}+M_{2,n},
	\end{align}
	for any $\lambda>0, v\geq 1$.
	Set $\lambda\coloneqq \eta \sqrt{\ln n/(n\rphi_{x}(h))}$ and $v\coloneqq C' (\ln n)^2$ where $\eta, C'>0$ are arbitrary.

	We start with the term $A_1$. Rewrite
 \begin{align*}
     M_{1,n}&=4\parent[\bigg]{1+\frac{\eta^2 n\rphi_x(h)}{C' S_{n,\ell,q-2}^2(x)\ln n}}^{-v/2}=4\parent[\bigg]{1-\frac{\eta^2 n\rphi_x(h)}{C' S_{n,\ell,q-2}^2(x)\ln n+\eta^2 n\rphi_x(h)}}^{v/2}\\
     &\coloneqq 4(1-t_n)^{v/2}.
 \end{align*}
 Inspecting the sequence $\{t_n\}_{n\in\N}$ with the help of  Lemma \ref{l3}, we obtain that
	\begin{equation}\label{eq29_}
		t_n\leq  \frac{\eta^2 n \rphi_{x}(h)}{C' Cn\rphi_x(h) \ln n +\eta^2 n\rphi_{x}(h)}= \frac{1}{\ln n} \frac{\eta^2}{\parent{C'c+\eta^2}} \leq \frac{1}{\ln n},
	\end{equation}
	and 
	\begin{equation}\label{eq30_}
		t_n \geq  \frac{\eta^2 n \rphi_{x}(h)}{C' cn\rphi_x(h) \ln n +\eta^2 n\rphi_{x}(h)}\geq \frac{1}{\ln n+1}\geq \frac{1}{2\ln n},
	\end{equation}
	by choosing $\eta^2=C'C$ and for all $n$ sufficiently large. From \eqref{eq29_}, $0\leq t_n\leq (\ln n)^{-1}$ for all $n$ large enough which implies that $t_n\to 0, n\to\infty$. It is well known that the first order Taylor 
 expansion of $g(t_n)\coloneqq\ln(1-t_n)$ for $t_n\to 0, n\to\infty$, satisfies $g(t_n)=g(0)+g'(0)(t_n-0)+o(t_n)=-t_n+o(t_n)$. Hence $\ln (1-t_n)\overset{a}{\approx}-t_n$.\footnote{Note that the result $\ln (1-t_n)\overset{a}{\approx}-t_n$ is not guaranteed without the lower bound in Lemma \ref{l3}.}. Clearly, $(v/2)\ln (1-t_n)\overset{a}{\approx}-(v/2)t_n$ also holds, and since the exponential function is continuous on $\R$, inequality  \eqref{eq30_} implies
	\begin{equation}\label{eq31_}
		(1-t_n)^{v/2}\overset{a}{\approx} \exp\chave[\Big]{\frac{-t_nv}{2}}\leq \exp\chave[\Big]{-\frac{1}{2\ln n}\frac{C'(\ln n)^2}{2}}=n^{-C'/4},
	\end{equation}
	and thus 
	\begin{equation}\label{eq32_}
		M_{1,n}=O(n^{-C'/4}).
	\end{equation}
	
	Next, we focus on the term $A_2$. We have that
	\begin{equation}\label{eq33_}
		A_2\leq C\frac{n}{(\ln n)^{2}}\parent[\bigg]{\frac{(\ln n)^3}{ n\rphi_{x}(h) }}^{\frac{r(a+1)}{2(a+r)}}\leq Cn^{1-\frac{1}{2}\frac{(a+1)r}{a+r}}(\ln n)^{-2+\frac{3}{2}\frac{(a+1)r}{a+r}}\rphi_{x}(h)^{-\frac{1}{2}\frac{(a+1)r}{a+r}}.
	\end{equation}
	Define $g_1(r)=(a+1)r/(a+r)$ given $a=3+\delta$, in view of \textbf{A8}. Then $g_1$ is a positive monotone increasing function on $\R_+$ such that $\lim_{r\to\infty}g_1(r)=a+1=4+\delta$. By \textbf{A8}, there is $\Delta>0$ such that $\epsilon\coloneqq \delta-\Delta>0$. Then $(4+\delta)-g_1(r)<\epsilon$, or equivalently, $4+\Delta<g_1(r)$ for any $r$ sufficiently large. Thus, from \eqref{eq33_} and \textbf{A8},
	\begin{align}\notag
		A_2&\leq C n^{1-(4+\Delta)/2}(\ln n)^{-2+3(a+1)/2}\rphi_{x}(h)^{-(a+1)/2}&\\\label{eq34_}
		&= \frac{C}{n^{1+\Delta/2}(\ln n)^2}\colc[\bigg]{\frac{(\ln n)^3}{\rphi_{x}(h)}}^{(a+1)/2}\leq \frac{C}{n^{1+\Delta/2}(\ln n)^2}n^{\Delta/2} = \frac{C}{n(\ln n)^2},
	\end{align}
	for all $n$ and $r$ large enough. As $C'>0$ can be chosen arbitrarily large in \eqref{eq32_}, a suitable choice of $C'$ implies that $A_1=o\parent[\big]{1/(n(\ln n)^2)}$. 
	Therefore, by combining \eqref{eq28_}, \eqref{eq32_} and \eqref{eq34_}, we have that\footnote{See Theorem 3.29 of \cite{rudin}.}
	\begin{equation*}
		\sum_{n=1}^{\infty}P\parent[\big]{\abs{S_{q,\ell}-ES_{q,\ell}}>\lambda}=\sum_{n=1}^{\infty}P\parent[\bigg]{\abs[\bigg]{\sum_{i=1}^{n}Z_i^{(q,\ell)}}>4\lambda n\rphi_{x}(h)}\leq C \sum_{n=1}^{\infty}\frac{1}{n(\ln n)^2}<\infty, 
	\end{equation*}
	with  $\lambda\coloneqq\eta \sqrt{\ln n/(n\rphi_{x}(h))}$ which shows the desired result.
	
	The proof of (d) is omitted since we can proceed along the same lines  as in the proof of Lemma \ref{l3} (see Section 2 of the Supplementary Material). It is worth noting that $1/(n\rphi_{x}(h))=o\parent[\big]{\sqrt{\ln n/(n\rphi_{x}(h))}}$.
	
	With results (a), (c) and (d) in hand, it follows that
	\begin{equation*}
		S_{2,1}S_{4,0}-E(	S_{2,1}S_{4,0})=\Oac\parent[\Bigg]{\sqrt{\frac{\ln n}{n\rphi_{x}(h)}}}.
	\end{equation*}
	The same result holds for $S_{3,1}S_{3,0}-E(	S_{3,1}S_{3,0})$. Thus, from (b),
	\begin{equation*}
		m_1-Em_1=\Oac\parent[\Bigg]{\sqrt{\frac{\ln n}{n\rphi_{x}(h)}}}O\parent[\big]{\rphi_{x}(h)^{1-2p_\text{max}}}=\Oac\parent[\Bigg]{\sqrt{\frac{\ln n}{n\rphi_x(h)^{4p_\text{max}-1}}}}.
	\end{equation*}
	In particular, if we set $\varphi=1$, then we get
	\begin{equation*}
		m_0-Em_0=m_0-1=\Oac\parent[\Bigg]{\sqrt{\frac{\ln n}{n\rphi_x(h)^{4p_\text{max}-1}}}},
	\end{equation*}
	proving \eqref{eq25_}.
	
	Now we focus on the asymptotic orders in probability in \eqref{eq26_} and \eqref{eq27_}. As the results (a), (b) and (d) are already proven and does not depend on \textbf{A6}(i), it is sufficient to show that 
	$S_{q,\ell}-ES_{q,\ell}=O_p\parent[\Big]{\sqrt{\ln n/(n\rphi_{x}(h))}\ }$. It can be easily obtained from the fact that
	\begin{align*}
		\var\parent{S_{q,\ell}}&=E\colc[\bigg]{\parent[\bigg]{\frac{1}{n\rphi_{x}(h)}\sum_{i=1}^n \Lambda_i^{(q-2,\ell)}(x)}^2}\\
		&\leq \frac{1}{(n\rphi_{x}(h))^2}\sum_{i,j=1}^n\abs[\Big]{\cov\parent[\big]{\Lambda_i^{(q-2,\ell)}(x),\Lambda_j^{(q-2,\ell)}(x)}}=\frac{S^2_{n,\ell,q-2}(x)}{(n\rphi_{x}(h))^2}.
	\end{align*}
	By the result \eqref{eq17_} of Lemma \ref{l3}, $\var\parent{S_{q,\ell}}=O(1/(n\rphi_{x}(h)))$, and so, using the Chebychev's inequality we have that
	\begin{align*}
		P\parent[\bigg]{\abs{S_{q,\ell}-ES_{q,\ell}}\geq \frac{\epsilon}{\sqrt{n\rphi_{x}(h)}}}\leq \frac{\var\parent{S_{q,\ell}} n\rphi_{x}(h)}{\epsilon}\leq \frac{C}{\epsilon},\quad \forall \epsilon>0,
	\end{align*}
	which implies the desired result.\qed

\begin{remark}
	The proofs of Lemmas 2 and 5 of \cite{leulmi},  require Taylor approximations $\ln(1+x)=x-x^2/2+o(x^2)$, as $x\to 0$, in order to bound the terms implied by the Fuk-Nagaev's inequality, where $x$ is related to the term $A_1$ in \eqref{eq28_}. 
	However, to ensure that  $x\to 0$ as $n\to\infty$ the result $S_{n,\ell,q-2}^2=O(n\rphi_x(h))$ stated in their Lemma A.2 is not sufficient. To be on the safe side, we provide a stronger and sufficient result in Lemma \ref{l3} (as well as in  Lemmas \ref{l5} and \ref{l6}). 
\end{remark}

\begin{prop}\label{p4}
	If the assumptions \textbf{H1}-\textbf{H8} hold, then
	\begin{equation}
		\sup_{x\in S}\abs{m_1 (x)-Em_1(x)}=\Oac\parent[\bigg]{\sqrt{\frac{\ln n}{n\rphi_x(h)^{4p_{\max}-1}}}},
	\end{equation}
	and 
	\begin{equation}
		\sup_{x\in S}\abs{m_0 (x)-1}=\Oac\parent[\bigg]{\sqrt{\frac{\ln n}{n\rphi_x(h)^{4p_{\max}-1}}}}.
	\end{equation}
\end{prop}
\noindent\textbf{Proof of Proposition \ref{p4}}
	As argued in the proof of Proposition \ref{p3}, it is sufficient to show that, uniformly on $x\in S$, for $q\in \{2,3,4\}$ and $\ell\in\{0,1\}$, 
	\begin{enumerate}
		\item[(a)] $ES_{q,\ell}=O(1)$;
		\item[(b)] $Q=O(\rphi_{x}(h)^{1-2p_\text{max}})$;
		\item[(c)] $S_{q,\ell}-ES_{q,\ell}=\Oac\parent[\Big]{\sqrt{\ln n/(n\rphi_{x}(h))}\ }$;
		\item[(d)] $\cov(S_{2,1},S_{4,0})=o\parent[\Big]{\sqrt{\ln n/(n\rphi_{x}(h))}\ }, \cov(S_{3,1},S_{3,0})=o\parent[\Big]{\sqrt{\ln n/(n\rphi_{x}(h))}\ },$
	\end{enumerate}
	where \begin{equation*}
		S_{q,\ell}\coloneqq\frac{1}{n\rphi_{x}(h)}\sum_{i=1}^n\frac{K_i\beta_i^{q-2}\varphi_i^\ell}{h^{q-2}} \text{ and } Q\coloneqq \frac{\parent[\big]{nh\rphi_{x}(h)}^2}{\Gamma(x)}.
	\end{equation*}
	
	Items (a), (b) and (d) follow from similar arguments to those used to prove Proposition 3.
	
	It remains to show (c). For $x\in S$, set $j(x)\coloneqq \argmin_{j\in[N_{r_n}(S)]}d(x,x_j)$. Then 
	\begin{align*}
		\sup_{x\in S}\abs{S_{q,\ell}(x)-ES_{q,\ell}(x)}&\leq \sup_{x\in S}\abs{S_{q,\ell}(x)-S_{q,\ell}(x_{j(x)})}+\sup_{x\in S}\abs{S_{q,\ell}(x_{j(x)})-ES_{q,\ell}(x_{j(x)})}\\
		&+\sup_{x\in S}\abs{ES_{q,\ell}(x_{j(x)})-ES_{q,\ell}(x)}\\
		&\coloneqq A_1+A_2+A_3.
	\end{align*}
	
	We start with the term $A_2$. Using the monotonicity and the subadditivity of the measure $P$, it holds that for any $\lambda>0$
	\begin{align*}
		P(A_2>\lambda)&=P\parent[\Big]{\max_{j\in[N_{r_n}(S)]}\abs{S_{q,\ell}(x_{j})-ES_{q,\ell}(x_{j})}>\lambda}\\
		&\leq \sum_{j=1}^{N_{r_n}(S)} P\parent{\abs{S_{q,\ell}(x_{j})-ES_{q,\ell}(x_{j})}>\lambda}\\
		&\leq N_{r_n}(S)\max_{j\in[N_{r_n}(S)]} P\parent{\abs{S_{q,\ell}(x_{j})-ES_{q,\ell}(x_{j})}>\lambda}.
	\end{align*}
	The application of the Fuk-Nagaev's inequality gives 
	\begin{align*}
		N_{r_n}(S)P\parent{\abs{S_{q,\ell}(x_{j})-ES_{q,\ell}(x_{j})}>\lambda}\leq A_{2,1}+A_{2,2}
	\end{align*}
	where $$A_{2,1}=CN_{r_n}(S)\parent[\bigg]{1+\frac{(\lambda n\rphi(h))^2}{vS^2_{n,\ell,q-2}}}^{-v/2} \textup{ and } A_{2,2}=CN_{r_n}(S)\frac{n}{v}\parent[\bigg]{ \frac{n}{\lambda n\rphi(h)} }^{r(a+1)/(a+r)},$$
	for any $v\geq 1$ and $r\geq 2$. Set $\lambda\coloneqq \eta \sqrt{\ln n/(n\rphi(h))}$ and $v\coloneqq C'(\ln n)^2$ with $\eta,C'>0$ being arbitrary constants. Similar to what has been done for item (d) in the proof of Proposition \ref{p3}, and with the help of Lemma \ref{l5}, one can check that uniformly on $x\in S$ 
	$$A_{2,1}\leq C N_{r_n}(S)n^{-C'} \textup{ and } A_{2,2}\leq  CN_{r_n}(S)\frac{n}{(\ln n)^{2}}\parent[\bigg]{\frac{(\ln n)^3}{ n\rphi(h) }}^{d}$$ 
	where $d=r(a+1)/[2(a+r)]$.  
	Note that $N_{r_n}(S)\overset{a}{\approx} n^{C_0}$ for some $C_0>0$ from \textbf{H8}. Then $A_{2,1}=O(n^{C_0-C'})=O(n^{-1-\xi})$ for some $\xi>0$ as long as $C'>0$ is chosen suitably large.  On the other hand, since  geometric mixing rates imply arithmetic mixing rates for any $a>0$, we can pick $r=a=4(2+C_0)/(1-\Delta_1)-1>2$, implying $	d(\Delta_1-1)=-2-C_0$, to conclude that
	\begin{equation} 
		A_{2,2}\leq C\frac{n^{1+C_0}}{(\ln n)^{2}}\parent[\bigg]{\frac{(\ln n)^3}{ n\rphi_{x}(h) }}^{d}\leq C\frac{n^{1+C_0-d(1-\Delta_1)}}{(\ln n)^2}= \frac{C}{(\ln n)^2 n}.
	\end{equation}
	As $n^{-1-\xi}=o(1/(n(\ln n)^2))$, the term $A_{2,1}$ is dominated by $A_{2,2}$, and so, we have that
	
	\begin{equation*}
		\sum_{n=1}^{\infty}P\parent[\Bigg]{A_2>\eta \sqrt{\frac{\ln n}{n\rphi(h)}}\ }\leq C \sum_{n=1}^{\infty}\frac{1}{(\ln n)^2 n}<\infty. 
	\end{equation*}
	
	Next, we cope with the term $A_1$. Rewrite
	\begin{align*}
		A_1&=\frac{1}{n\rphi(h)h^{q-2}}\sup_{x\in S}\sum_{i=1}^n\abs{\varphi_i}^\ell K_i(x)1_{B(x,h)}(\rchi_i)\abs{\beta_i(x)^{q-2}-\beta_i(x_{j(x)})^{q-2}1_{B(x_{j(x)},h)}(\rchi_i)}\\
		&+\frac{1}{n\rphi(h)h^{q-2}}\sup_{x\in S}\sum_{i=1}^n\abs{\varphi_i}^\ell \abs{\beta_i(x_{j(x)})}1_{B(x_{j(x)},h)}(\rchi_i)\abs{K_i(x)1_{B(x,h)}(\rchi_i)-K_i(x_{j(x)})}.
	\end{align*}
	Put
	\begin{align*}
		R_{i,q,x}^1&=1_{B(x,h)}(\rchi_i)\abs{\beta_i(x)^{q-2}-\beta_i(x_{j(x)})^{q-2}1_{B(x_{j(x)},h)}(\rchi_i)},\\
		R_{i,q,x}^2&=1_{B(x_{j(x)},h)}(\rchi_i)\abs{K_i(x)1_{B(x,h)}(\rchi_i)-K_i(x_{j(x)})},
	\end{align*}
	and observe that 
	\begin{align*}
		R_{i,q,x}^1&=\left\{\begin{array}{ll}
			\abs{\beta_i(x)-\beta_i(x_{j(x)})}\abs[\big]{\sum_{k=0}^{q-3}\beta_i(x)^{q-3-k}\beta_i(x_{j(x)})^{k}}&,  \text{if } \rchi_i\in B(x,h)\cap B(x_{j(x)},h)\\
			\abs{\beta_i(x)}^{q-2}&,  \text{if } \rchi_i\in B(x,h)\setminus B(x_{j(x)},h)\\
			0&,  \text{if } \rchi_i\notin B(x,h)
		\end{array}\right. \\
		R_{i,q,x}^2&=\left\{\begin{array}{ll}
			\abs{K_i(x)-K_i(x_{j(x)})}&,  \text{if } \rchi_i\in B(x_{j(x)},h)\cap B(x,h)\\
			K_i(x_{j(x)})&,  \text{if } \rchi_i\in B(x_{j(x)},h)\setminus B(x,h)\\
			0&,  \text{if } \rchi_i\notin B(x_{j(x)},h)
		\end{array}\right. .
	\end{align*}
	Moreover, the triangle inequality implies that $\abs{d(x,\rchi_i)-d(x_{j(x)},\rchi_i)}\leq d(x,x_{j(x)})$. From these observations, we apply \textbf{H3} and \textbf{H4} to obtain that
	\begin{align}\notag
		A_1&\leq C\frac{r_n}{nh\rphi(h)}\sup_{x\in S}\sum_{i=1}^n\abs{\varphi_i}^\ell\chave[\big]{1_{B(x_{j(x)},h)\cap B(x,h)}(\rchi_i) +1_{B(x,h)\setminus B(x_{j(x)},h)}(\rchi_i) \\\notag
			&+ 1_{B(x_{j(x)},h)\setminus B(x,h)}(\rchi_i) }\\\notag
		&=\frac{C}{n\rphi(h)}\sup_{x\in S}\sum_{i=1}^n \frac{r_n}{h} \abs{\varphi_i}^\ell1_{B(x_{j(x)},h)\cup B(x,h)}(\rchi_i)\\\notag
		&\coloneqq \frac{C}{n\rphi(h)}\sup_{x\in S}\sum_{i=1}^nT_i^\ell(x)\\\notag &=\frac{C}{n\rphi(h)}\sup_{x\in S}\sum_{i=1}^n\colc[Big]{T_i^\ell(x)-ET_i^\ell(x)} +\frac{C}{n\rphi(h)}\sup_{x\in S}\sum_{i=1}^nET_i^l(x)\\ \label{eq38_}
		&\coloneqq A_{1,1}+A_{1,2}
	\end{align}
	By Lemma \ref{l6}, we have that $A_{1,2}=O(r_n/h)$. The application of  Fuk-Nagaev's inequality, in a similar way as did for the term $A_2$ with 
	$$\lambda=\eta\parent[\bigg]{\frac{r_n}{h}}^2\sqrt{\frac{\ln n}{n\rphi(h)}} \textup{ and } v=C'\parent[\bigg]{\frac{r_n}{h}}^{-2}(\ln n)^2$$
	leads to $A_{1,2}=\Oac\parent[\big]{(r_n/h)^2\sqrt{\ln n/(n\rphi(h))}}$. Therefore
	\begin{align*}
		A_1&=O\parent[\bigg]{\frac{r_n}{h}}+\Oac\parent[\Bigg]{\parent[\bigg]{\frac{r_n}{h}}^2\sqrt{\frac{\ln n}{n\rphi(h)}}\ }=\Oac\parent[\Bigg]{\frac{r_n}{h}\colc[\bigg]{1 + \underbrace{\frac{r_n}{h} \parent[\bigg]{ \frac{\ln n}{n\rphi(h)} }^{1/2}}_{=o(1)} } }\\
		&=\Oac\parent[\bigg]{\frac{r_n}{h}}=\Oac \parent[\bigg]{\sqrt{\frac{\ln n}{n\rphi(h)}}\ },
	\end{align*}
	using the facts that $\rphi(h)=\lim_{a\to 0}\int_a^h\rphi'(h)\leq C h$ from \textbf{H1}, and that 
	\begin{align*}
		\frac{r_n}{h}\sqrt{\frac{n\rphi(h)}{\ln n}}\leq C \sqrt{\frac{\ln n}{n\rphi(h)}}=o(1),
	\end{align*}
	from \textbf{H5} and \textbf{H8}. 
	The last term $A_3$ satisfies
	\begin{align*}
		A_3&\leq \sup_{x\in S} E\abs{S_{q,\ell}(x)-S_{q,\ell}(x_{j(x)})}\leq \frac{C}{n\rphi(h)}\sum_{i=1}^n \sup_{x\in S} ET_i^\ell(x)\\
		&\leq C\frac{r_n}{h}\leq C \sqrt{\frac{\ln n}{n\rphi(h)}}.
	\end{align*}
	This completes the proof. \qed

\section*{Appendix B: Main proofs}

\textbf{Proof of Theorem \ref{teo1}} Let $m_\ell(x)=(1/\Gamma(x))\sum_{i\neq j}^n w_{i,j}(x) \varphi_j^\ell$ where $\Gamma(x)=\sum_{i\neq j}^n E(w_{i,j}(x))$, for $\ell\in\N$. Then
\begin{align*}
	\hat m_\varphi(x)-m_\varphi(x)&=\frac{1}{m_0(x)}\chave{\colc{m_1(x)-Em_1(x)} -\colc{m_\varphi(x)-Em_1(x)}}\\
	&-\frac{m_\varphi(x)}{m_0(x)}\colc{m_0(x)-1}.
\end{align*}
Denote $a_n= \sqrt{\ln n/(n\rphi_x(h)^{4p_\text{max}-1})}$. From Propositions \ref{p1} and \ref{p3}, it follows that\footnote{For more details, see  Propositions 5-6 of the Supplementary Material.}
\begin{align*}
	\hat m_\varphi(x)-m_\varphi(x)&=\colc[\big]{1+\Oac\parent[\big]{a_n}   }\chave[\big]{\Oac\parent[\big]{a_n}+O(h^b)}\\
	&=\Oac\parent[\big]{a_n}+O(h^b).
\end{align*}\qed

\vspace{1em}
The proofs of Corollary \ref{coro1} and Theorem \ref{teo2} are similar to that of Theorem \ref{teo1}, and thus omitted.
\vspace{1em}

\textbf{Proof of Corollary \ref{coro2}} Under independence, $$\Psi_{x,i,j}(h)=\colc[\big]{\rphi_{x,i}(th,h)\rphi_{x,j}(wh,h)}^{1/2+p_{i,j}}, \ \forall i,j\in[n],$$ with $p_{i,j}=p_{1,i,j}=p_{2,i,j}=1/2$. Thus $p_{\text{max}}=1/2$, and the result follows immediately from Theorem \ref{teo1}. \qed

\section*{Appendix C: Notes on previous studies}

This work is an extension of the articles of \cite{barrientos} and \cite{leulmi} (hereafter, ``BM'' and ``LM'', respectively).  BM studied the local linear estimator, discussed in this paper, for independent and identically distributed functional data. Subsequently, LM allowed the data to be weakly dependent. Unfortunately, some conditions of the latter authors seem to be too  restrictive and their derived asymptotics lacks rigor. Such issues will be discussed in the following.  

From now on, consider a sequence $\{(Y_i,\rchi_i)\}_{i\in n}$ that is equally distributed  as  $(Y,\rchi)$ and weakly dependent ($\alpha$-mixing) such that $Y_i=m(\rchi_i)+\epsilon_i,\ i\in[n],$ with $E(\epsilon_i|\rchi_i)=0$.

\vspace{1em}
\noindent
\textit{Issues related to the asymptotic results}

 Let $f_0:\R\to\R$ be a measurable function. In general, we cannot conclude that $$E(f_0(\rchi_i)f_0(\rchi_j))=E(f_0(\rchi_1)f_0(\rchi_2)),\ \forall i\neq j.$$ For instance, put $f_0(\rchi_i,\rchi_j)=K_i\beta_iK_j\beta_j$ with  $K=1_{[0,1]}$ being the uniform kernel and $\beta=d$. Then
\begin{equation*}
    E(K_i\beta_iK_j\beta_j)=h^2 \int_{[0,1]^2}uv \ dP'_{ij}(u,v), \ \forall i\neq j.
\end{equation*}
where $P'_{i,j}$ is the probability distribution of $(d(\rchi_i,x)/h,d(\rchi_j,x)/h)$. Each joint distribution $P'_{i,j}$ is determined
 depending on how  $d(\rchi_i,x)/h$ and $d(\rchi_j,x)/h$ are related. Due to this fact, we cannot conclude that all $E(K_i\beta_iK_j\beta_j), i\neq j,$ are equal to $E(K_1\beta_1K_2\beta_2)$ if the data is dependent. 
However, this equivalence is used to prove the main results of LM.  To cite an example, consider their proof of Lemma A.2 (which is used to prove their Lemma 2). In view of their assumption (H5b) and Lemma A1(ii)(iii) which are stated only in terms of $\rchi_1$ and $\rchi_2$, the following inequality is used:
$$\sum_{(i,j)\in S_1} E\parent[\big]{K_i\beta_i^k K_j\beta_j^k}\leq  \#S_1  E\parent[\big]{K_1\beta_1^k K_2\beta_2^k}=O(nm_n \rphi_{x,1}(h)^{1+d}), \quad k\in\{0,2\}, \ 0<d\leq 1,$$
where $S_1=\{(i,j): 1\leq \abs{i-j}\leq m_n\}$ with $m_n$ being a diverging sequence. Because there is no reason to  $E\parent[\big]{K_1\beta_1^k K_2\beta_2^k}$ be the greatest term in the summation, we are left to check the equality. As discussed before, the equality does not need to hold. 

Another example can be found in their proof of Lemma 1, where the arguments of BM are replicated. However, the proof of the latter authors uses results that require i.i.d. data: (i) $E(w_{i,j}(x))=E(w_{1,2}(x)), \forall i\neq j$; and (ii) $E(w_{1,2}(x) Y_2)=E(w_{1,2}(x) m(\rchi_2))$. Note that item (i) holds for i.i.d. data as shown below, for all $i\neq j$,
\begin{align*}
    E(w_{i,j})=&E(\beta_i^2 K_i K_j)-E(\beta_i K_i \beta_j K_j) \\
    \overset{indep.}{=} &E(\beta_i^2 K_i) E(K_j)-E(\beta_i K_i)E( \beta_j K_j)\\
    \overset{ident.}{=} &E(\beta_1^2 K_1) E(K_2)-E(\beta_1 K_1)E( \beta_2 K_2)=E(w_{1,2}).
\end{align*}
From the previous discussion, without the assumption of independence this equality does not need to hold. Now, if $\rchi_i$ is independent of $\rchi_j$, $i\neq j$, then (ii) can be verified using the Law of Iterated Expectations,
\begin{align*}
    E(w_{1,2} Y_2)=E(E(w_{1,2} Y_2|\rchi_2))\overset{indep.}{=}E(w_{1,2} \  m(\rchi_2)).
\end{align*}
For dependent data, we need to take the expectation conditioned to $(\rchi_1,\rchi_2)$,
\begin{align*}
    E(w_{1,2} Y_2)=E(w_{1,2} E(Y_2|(\rchi_1,\rchi_2))=E(w_{1,2}( m(\rchi_2)+E(\epsilon_2|(\rchi_1,\rchi_2)) ).
\end{align*}\
To ensure that $E(\epsilon_2|(\rchi_1,\rchi_2))=0$ using the assumption  $E(\epsilon_2|\rchi_2)=0$, the additional requirement that  the error $\epsilon_2$ is independent of $\rchi_1$ is needed.

By Fuk-Nagaev's inequality, LM derived the term $A_1(x)$ in their proof of Lemma 2. This term is equivalent to  $M_{1,n}$ which appears in the proof of Proposition \ref{p3} (Appendix A). LM bound this term by applying the Taylor expansion $\ln(1+x)=x-x^2/2+o(x^2)$ where $x$ tends to zero. In their case,
\begin{equation*}
    x\coloneqq x_n=\frac{\eta^2 n\rphi_x(h)}{ S_{n,\ell,q-2}^2(x)\ln n},
\end{equation*}
making use of our notations in the proof of Proposition \ref{p3}. By the hypothesis of LM, $n\rphi_x(h)/\ln n\to \infty$ as $n\to\infty$. Since $\{x_n\}_{n\in\N}$ is a positive sequence,  we cannot conclude that $x_n=o(1)$ without giving a suitable positive lower bound for $S_{n,\ell,q-2}^2(x)$ (to ensure that this term diverges to infinity faster than $n\rphi_x(h)/\ln n$).

The same issues pointed above also happen in their calculations related to the uniform convergence.

\vspace{1em}
\noindent
\textit{Weakening the assumptions}

The framework of LM requires that the kernel function $K$ is bounded below by a positive constant on its support $[0,1]$, which can be seem in their assumptions (H4) and (U4). However,  popular choices like the triangle, quadratic or cubic kernel functions satisfy $K(1)=0$, and thus are excluded from their analysis. Assumptions \textbf{A5} and \textbf{H4} (in Sections \ref{sec3.2} and \ref{sec3.3}, respectively) allow for both types of functions. 

 In view of our previous discussion on the asymptotics derived by LM, their assumption (H5b) that relates the local joint cumulative distribution function (CDF) and its marginal CDFs (respectively, $\Psi_{x,1,2}(h)$ and $\rphi_{x,1}(h)\rphi_{x,2}(h)$, under our notations) should be stated not only in terms of $(\rchi_1,\rchi_2)$, but also $(\rchi_i,\rchi_j), \forall i\neq j$. That is,
 
\vspace{1em}\noindent
(H5b)' There exist $0<d\leq 1, C>0, C'>0$ such that $C'[\rphi_{x,1}(h)]^{1+d}<\Psi_{x,i,j}(h)\leq C[\rphi_{x,1}(h)]^{1+d}, \forall i,j\in[n]:i\neq j.$

 \vspace{1em}
 
 However, it is of interest to ask whether it is compatible to assume that $\{(Y_i,\rchi_i)\}$ is strongly mixing with arithmetic rate $a>3$ and $\Psi_{x,i,j}(h)=\Theta\parent[\big]{\colc{\phi_{x,1}(h)}^{1+d}}$ for some $d\neq 1$ and all $i\neq j$. Consider the sets $S_s=\{(i,j): 1\leq \abs{i-j}< b_n\}$ and $S_l=\{(i,j): b_n\leq \abs{i-j}\leq n-1\}$. Then (H5b)' applies to all joint CDFs $\{\Psi_{x,i,j}(h)\}_{(i,j)\in[n]^2:i\neq j}=\{\Psi_{x,i,j}(h)\}_{(i,j)\in S_s}\cup \{\Psi_{x,i,j}(h)\}_{(i,j)\in S_l}\coloneqq \Psi_{S_s}\cup \Psi_{S_l}$. Proposition 3 in the Supplementary Material shows that $\Psi_{S_l}=\Theta (\rphi_{x,1}(h)^2)$ for $b_n=1/\rphi_{x,1}(h)$. In this case, note that the set $S_l$ is nonempty for every $n$ large enough since it contains at least the elements $\{(i,j):\abs{i-j}=n-1\}$.\footnote{If $(i,j)$ is such that $\abs{i-j}=n-1$, then it also satisfies $1/\rphi_{x,1}(h)\leq \abs{i-j}$ for $n$ large because $0\leq n\rphi_{x,1}(h) -\rphi_{x,1}(h)-1$ holds by the hypothesis that $n\rphi(h)\to\infty$, $n\to\infty$.} This is sufficient to show that there cannot exist $d\neq 1$ such that $\Psi_{x,i,j}(h)=\Theta\parent[\big]{\colc{\phi_{x,1}(h)}^{1+d}}$ for all $i\neq j$. Thus (H5b)' would be better written if we explicitly consider $d=1$. Unfortunately, (H5b)' is too restrictive for strongly mixing data in the sense that, asymptotically, $\Psi_{x,i,j}(h)=\Theta (\rphi_{x,1}(h)^2), \forall i\neq j$, is not much different from the case where data is independent ($\Psi_{x,i,j}(h)=\rphi_{x,1}(h)^2$).

\end{document}


\pagestyle{myheadings}
\markboth{}{Matsuoka and Torrent}

\thispagestyle{empty}
{\centering
\Large{\bf Supplementary material for ``Strong consistency of the local linear estimator for a generalized regression function with dependent functional data".} \vspace{.5cm}\\
\normalsize{ {\bf 
Danilo H. Matsuoka${}^{\mathrm{a}}$,\let\thefootnote\relax\footnote{\hskip-.3cm$\phantom{s}^\mathrm{a}$Research Group of Applied Microeconomics - Department of Economics, Federal University of Rio Grande.} Hudson da Silva Torrent${}^\mathrm{b}$\let\thefootnote\relax\footnote{\hskip-.3cm$\phantom{s}^\mathrm{b}$Mathematics and Statistics Institute - Universidade Federal do Rio Grande do Sul.
}
 \\
\let\thefootnote\relax\footnote{E-mails: danilomatsuoka@gmail.com (Matsuoka); hudsontorrent@gmail.com (Torrent)}
\vskip.3
cm}}
}
\section*{Introduction}

In this supplement, we provide details that are omitted in the paper. Section \ref{sec:2} presents the proofs of Lemmas stated in  Appendix A and  Section \ref{sec:3} clarifies some basic aspects and gives some examples where  the assumptions made in the text should hold.
 \vspace{1em}

\section{Proofs of Lemmas}\label{sec:2}

 In this section, we will use all the notations introduced in the paper.
 
 \vspace{1em}
 
	\noindent\textbf{Proof of Lemma 1}.
	Let $i,j\in[n]$ be given. Denote the indicator function and the pushforward measure induced by $(d(x,\rchi_i)/h,d(x,\rchi_j)/h)$, respectively, by $I(\cdot)$ and $P_{i,j}'$. 
	
	$(i)$ From \textbf{A3} and \textbf{A5},
	\begin{align*}
		E(K_i^q\abs{\beta_i}^\ell)&
         \leq CE\colc[\big]{K_i^q d(x,\rchi_i)^\ell} 
        =CE\colc[\big]{ I(0\leq d(x,\rchi_i)\leq h)K_i^q d(x,\rchi_i)^\ell} \\
        &\leq C h^l P(0\leq d(x,\rchi_i)\leq h)= C h^\ell \rphi_{x,i}(h).
	\end{align*}

	$(ii)$  Using \textbf{A3} and \textbf{A5}
	\begin{align*}
		E(K_i K_j \abs{\beta_i}^{\ell_1}\abs{\beta_j}^{\ell_2})&
  \leq CE\colc[\big]{I(0\leq d(x,\rchi_i)\leq h,0\leq d(x,\rchi_j)\leq h)K_iK_jd(x,\rchi_i)^{\ell_1}d(x,\rchi_j)^{\ell_2}} \\
       &\leq C h^{\ell_1+\ell_2}P\parent[\big]{0\leq d(x,\rchi_i)\leq h,0\leq d(x,\rchi_j)\leq h} \\
       &= C h^{\ell_1+\ell_2} \Psi_{x,i,j}(h).
	\end{align*}

	$(iii)$ Firstly, let the kernel function $K$ be of type \textbf{A5}(II). Using \textbf{A3}, the Fundamental Theorem of Calculus and the Fubini's Theorem, we have that
	\begin{align*}
		E(K_iK_j\beta^2_i)&\geq c_3h^2\int_{[0,1]^2}u^2K(u)K(z)dP'_{i,j}(u,z)\\
		&=c_3h^2\int_{[0,1]^2}
		\colc[\bigg]{\int_{0}^{u}\frac{d}{dt}(K(t)t^2)dt}
		\colc[\bigg]{\int_{0}^{z}K'(w)dw+K(0)}dP'_{i,j}(u,z)\\
		&=c_3h^2\chave[\bigg]{\iint_{0}^{1}\int_{[t,1]\times[w,1]}dP'_{i,j}(u,z) \frac{d}{dt}(K(t)t^2)K'(w)dtdw\\
			&+K(0)\int_{0}^{1}\int_{[t,1]\times[0,1]}dP'_{i,j}(u,z) \frac{d}{dt}(K(t)t^2)dt
		}\\
		&=c_3h^2\chave[\bigg]{\iint_{0}^{1}
			\Psi_{x,i,j}(th,h,wh,h)
			\frac{d}{dt}(K(t)t^2)K'(w)dtdw\\
			&+K(0)\int_{0}^{1}\Psi_{x,i,j}(th,h,0,h) \frac{d}{dt}(K(t)t^2)dt
		}.
	\end{align*}
	By hypothesis, $K(0)=-\int_{0}^{1}K'(w)dw$. Note that, for all $t,w\in[0,1]$, the $P$-measurable sets 
 \begin{align*}
 E_1&=\{th\leq d(x,\rchi_i)\leq h, wh\leq d(x,\rchi_i)\leq h\},\\
E_2&=\{th\leq d(x,\rchi_i)\leq h, 0\leq d(x,\rchi_i)\leq h\}
 \end{align*}
 satisfy $E_1\cap E_2=E_1$, $E_2\setminus E_1=\{th\leq d(x,\rchi_i)\leq h, 0\leq d(x,\rchi_i)\leq wh\}$ and $(E_1\cap E_2)\cup (E_2\setminus E_1)=E_2$. As $E_1\cap E_2$ and $E_2\setminus E_1$ are disjoint sets, it holds that
 $P(E_2)=P(E_1)+P(E_2\setminus E_1)$, or equivalently,
 $$\Psi_{x,i,j}(th,h,0,h)-\Psi_{x,i,j}(th,h,wh,h)=\Psi_{x,i,j}(th,h,0,wh).$$
 Then, using \textbf{A10} and the previous considerations,
	\begin{align*}
		E(K_iK_j\beta^2_i)&\geq -c_3h^2\iint_{0}^{1}\Psi_{x,i,j}(th,h,0,wh) \frac{d}{dt}(K(t)t^2)K'(w)dtdw\\
		&> c_3c_{10}h^2\Psi_{x,i,j}(h)
	\end{align*}
	for  all $n$ sufficiently large.
	In view of the above arguments, if the function $K$ is of type \textbf{A5}(I), then, for all $n$ large enough,
	\begin{align*}
		E(K_iK_j\beta^2_i)&\geq c_3 h^2\int_{[0,1]^2}u^2K(u)K(z)dP'_{i,j}(u,z)\geq c_3c_5h^2\int_{[0,1]^2}u^2K(u)dP'_{i,j}(u,z) \\
		&= c_3c_5h^2\int_{0}^{1}\Psi_{x,i,j}(th,h,0,h)
		\frac{d}{dt}(K(t)t^2)dt>c_3c_5c_{10}h^2 \Psi_{x,i,j}(h).
	\end{align*}

	$(iv)$ Denote by $P'_i$ the probability induced by $d(x,\rchi_i)/h$. From \textbf{A3}, \textbf{A5} and \textbf{A6}(i), it follows that
	\begin{align*}
		E(K_i^2\beta_i^\ell)&\geq ch^\ell\int_{[0,1]}K(u)^2u^\ell dP'_i(u)\geq c h^\ell K(\epsilon_1)\int_{[0,\epsilon_1]}K(u)u^\ell dP'_i(u)\\
		&=ch^\ell\int_{0}^{\epsilon_1}\rphi_{x,i}(th,\epsilon_1 h)
		\frac{d}{dt}(K(t)t^\ell )dt >c h^\ell \rphi_{x,i}(h),
	\end{align*}
	for all $\ell \in\{2,4\}$ and some $0<\epsilon_1<1$. Now we turn to the case where $\ell =0$. For $K$ being of type (I) in A5, the result is straightforward. If $K$ is of type (II), then
	\begin{align*}
		E(K_i^2)&\geq K(\epsilon_0)\int_{[0,\epsilon_0]}K(u)dP'_i(u) \geq c\int_{[0,\epsilon_0]}\colc[\bigg]{\int_0^u K'(t)dt+K(0)}dP'_i(u)\\
		&\geq c\int_{[0,\epsilon_0]}\colc[\bigg]{\int_0^u K'(t)dt-\int_0^{\epsilon_0} K'(z)dz}dP'_i(u)\\
		&=c\int_{0}^{\epsilon_0}K'(t)\colc[\big]{\rphi_{x,i}(th,\epsilon_0 h)-\rphi_{x,i}(\epsilon_0 h)}dt\\
		&=-c\int_{0}^{\epsilon_0}K'(t)\rphi_{x,i}(th)dt
		\geq c\int_{0}^{\epsilon_0}\rphi_{x,i}(th)dt>c\rphi_{x,i}(h),
	\end{align*} 
	for some $\epsilon_0\in(0,1)$ and all $n$ sufficiently large. 
 
    $(v)$ If $\ell=0$, then $E(\abs{\varphi_i}^\ell|\rchi_i)=1$. Hence, let $\ell\geq 1$. For all integer $m>\ell$, the conditional Jensen's inequality implies that
    \begin{equation*}
    \colc[\big]{E\parent[\big]{\abs{\varphi_i}^\ell|\rchi_i}}^{m/\ell}\leq E\parent[\big]{\parent[\big]{\abs{\varphi_i}^\ell}^{m/\ell}|\rchi_i}=E\parent[\big]{\abs{\varphi_i}^{m}|\rchi_i}=\gamma_{m,i}(\rchi_i)
    \end{equation*}
    From the continuity of $\gamma_{m,i}$ on $x$, 
    \begin{equation*}
        \lim_{n\to+\infty}I(0\leq d(x,\rchi_i)\leq h)\gamma_{m,i}(\rchi_i)=\lim_{d(x,\rchi_i)\to 0}\gamma_{m,i}(\rchi_i)=\gamma_{m,i}(x).
    \end{equation*}
    Thus, $I(0\leq d(x,\rchi_i)\leq h)\gamma_{m,i}(\rchi_i)=\gamma_{m,i}(x)+o(1)$. Therefore, by \textbf{A4}, for all $n$ large enough
    \begin{equation*}
        I(0\leq d(x,\rchi_i)\leq h)E\parent[\big]{\abs{\varphi_i}^{l}|\rchi_i}\leq 1+\max_{s\in[n]}\gamma_{m,s}(x)+o(1)\leq C.
    \end{equation*}
 \qed

\vspace{1em}

    \noindent\textbf{Proof of Lemma 2}. 
    	Denote the cardinal of a set $A$ by $\#A$. Define, for $s\in\{1,\dotsc,n-\floor{a_n}\}$,
    \begin{equation}\label{e7}
    	E_s=\{(s,s+1),\dotsc,(s,s+\floor{a_n}), (s+\floor{a_n},s),\dotsc,(s+1,s)\}, 
    \end{equation}
    and, for $s\in\{n-\floor{a_n}+1,\dotsc,n-1\}$,
    \begin{equation}\label{e8}
    	G_s=\{(s,s+1),\dotsc,(s,n), (n,s),\dotsc,(s+1,s)\}. 
    \end{equation}
    Note that $\#E_s=2\floor{a_n}$ and $\#G_s=\#G_{n-\floor{a_n}+j}=2(\floor{a_n}-j)$  for $j\in\{1,\dotsc,\floor{a_n}-1\}$. Moreover, $S_1=\parent[\big]{\bigcup_{s=1}^{n-\floor{a_n}}E_s}\cup \parent[\big]{\bigcup_{s=n-\floor{a_n}+1}^{n-1}G_s}$ and the sets in \eqref{e7} and \eqref{e8} are pairwise disjoint. Therefore
    \begin{align*}
    	\#S_1&=\sum_{s=1}^{n-\floor{a_n}} \#E_s+\sum_{s=n-\floor{a_n}+1}^{n-1} \#G_s=2\floor{a_n}(n-2\floor{a_n})+\sum_{j=1}^{\floor{a_n}-1} 2(\floor{a_n}-j)\\
    	&=2\floor{a_n}n-\floor{a_n}^2-\floor{a_n}\overset{a}{\approx}2na_n.
    \end{align*} \qed

\vspace{1em}

\noindent\textbf{Proof of Lemma 3}. 
	Write
	\begin{align*}
		S_{n,\ell ,k}^2(x)&=\sum_{ S_1} \abs[\big]{\cov \parent[\big]{\Lambda_{i}^{(k,\ell )}(x),\Lambda_{j}^{(k,\ell )}(x)}}+\sum_{S_2} \abs[\big]{\cov \parent[\big]{\Lambda_{i}^{(k,\ell )}(x),\Lambda_{j}^{(k,\ell )}(x)}}\\
		&+\sum_{i=1}^n \var\parent[\big]{\Lambda_{i}^{(k,\ell )}(x)} \\
		&\coloneqq J_{1,n}+J_{2,n}+V_{n},
	\end{align*}
	where $S_1\coloneqq\{(i,j): 1\leq\abs{i-j}\leq a_n\}$ and  $S_2\coloneqq\{(i,j): a_n+1\leq\abs{i-j}\leq n-1\}$ are  subsets of  $[n]^2$, and $a_n$ is a positive real sequence diverging to infinity. The rate of  $a_n$ will be specified later on. 
	
	We start with the term $J_{1,n}$. Recall that \textbf{A4} implies that $E\parent[\big]{\abs{\varphi_i\varphi_j}^\ell\mid (\rchi_i,\rchi_j)}\leq 1+\sup_{i\neq j}E\parent[\big]{\abs{\varphi_i\varphi_j}\mid (\rchi_i,\rchi_j)}<C$. Since $E(\Lambda_{i}^{(k,\ell )}(x))=0$, $i\in [n]$, the Law of Iterated Expectations, Lemmas 1(i)(ii)(v), 2 and assumption \textbf{A9} give 
	\begin{align} \notag
		J_{1,n}&=\sum_{S_1}\abs[\big]{E\parent[\big]{\Lambda_{i}^{(k,\ell )}(x)\Lambda_{j}^{(k,\ell )}(x)}}\\ \notag
      &=\frac{1}{h^{2k}}\sum_{S_1}\abs[\big]{E\parent[\big]{K_i K_j (\beta_i\beta_j)^k (\varphi_i\varphi_j)^\ell}-E\parent[\big]{K_i\beta_i^k \varphi_i^\ell}E\parent[\big]{K_j\beta_j^k \varphi_j^\ell}}\\ \notag
		&\leq \frac{1}{h^{2k}}\sum_{S_1}\chave[\Big]{E\parent[\big]{K_i K_j \abs{\beta_i\beta_j}^k E\parent[\big]{\abs{\varphi_i\varphi_j}^\ell\mid (\rchi_i,\rchi_j)}}\\ \notag
			&+E\parent[\big]{K_i\abs{\beta_i}^k E(\abs{\varphi_i}^\ell|\rchi_i)}E\parent[\big]{K_j\abs{\beta_j}^k E(\abs{\varphi_j}^\ell|\rchi_j)}}\\ \notag
		&\leq \frac{C}{h^{2k}}\sum_{S_1}\chave[\Big]{E\parent[\big]{K_i K_j \abs{\beta_i\beta_j}^k}+E\parent[\big]{K_i\abs{\beta_i}^k }E\parent[\big]{K_j\abs{\beta_j}^k}}\\\notag
		&\leq C\sum_{S_1}\chave[\Big]{\colc[\big]{\rphi_{x,i}(h)\rphi_{x,j}(h)}^{1/2+p_{1,i,j}}+\rphi_{x,i}(h)\rphi_{x,j}(h)}\\ \label{ee10}
		&\leq C na_n \max_{s\in[n]}\rphi_{x,s}(h)^{1+2p_\text{min}}= C na_n \rphi_{x}(h)^{1+2p_\text{min}},
	\end{align}
	for all $n$ large enough, where $p_\text{min}=\min_{(i,j)\in [n]^2}p_{1,i,j}\in (1/4,1/2]$.
	
	Along the above lines, we also have that
	\begin{align*}
		V_n=\frac{1}{h^{2k}}\sum_{i=1}^n \chave[\big]{E\parent[\big]{K_i^2 \beta_i^{2k} \varphi_i^{2\ell}}-\colc[\big]{E\parent[\big]{K_i\beta_i^k \varphi_i^\ell}}^2} 
 \leq C\sum_{i=1}^{n}\chave[\big]{ \rphi_{x,i}(h)+\rphi_{x,i}(h)^2}\leq Cn\rphi_{x}(h).
	\end{align*}

	Next, we bound $J_{2,n}$ through the Davydov's inequality\footnote{See Theorem 1.1 of \cite{rio} or Corollary 1.1 of \cite{bosq}. }. By the Binomial Theorem, the Law of Iterated Expectations and Lemmas 1(i)(v), 
	\begin{align}\notag
		E\abs[\big]{\Lambda_{i}^{(k,\ell )}(x)}^q& = 
  \frac{1}{h^{kq}}E\abs[\big]{K_i \beta_i^k\varphi_i^\ell-E\parent[\big]{K_i\beta_i^k \varphi_i^\ell}}^q \leq \frac{1}{h^{kq}}E\chave[\big]{K_i \abs{\beta_i^k\varphi_i^\ell}+E\parent[\big]{K_i\abs{\beta_i^k \varphi_i^\ell}}}^q\\ \notag
  &=\frac{1}{h^{kq}}E\chave[\bigg]{\sum_{j=0}^{q} \binom{q}{j} \colc[\big]{K_i\abs{\beta_i^k\varphi_i^\ell} }^j \colc[\big]{E\parent[\big]{K_i\abs{\beta_i^k\varphi_i^\ell} }}^{q-j} } \\\notag
		&\leq  \frac{C}{h^{kq}}\sum_{j=0}^{q}  E\parent[\big]{K_i^j\abs{\beta_i}^{jk} E\parent[\big]{\abs{\varphi_i}^{\ell j}|\rchi_i}} \colc[\big]{E\parent[\big]{K_i\abs{\beta_i}^k E\parent[\big]{\abs{\varphi_i}^{\ell }|\rchi_i}}}^{q-j} \\ \notag
        & \leq \frac{C}{h^{kq}}\sum_{j=0}^{q}  h^{jk}\rphi_{x,i}(h) \colc[\big]{h^{k}\rphi_{x,i}(h)}^{q-j} \\ 
        \label{eap2}
		&= C\sum_{j=0}^{q} \rphi_{x,i}(h)^{1+q-j} \leq  C\max_{j\in\{0,\dotsc,q\}} \rphi_{x,i}(h)^{1+q-j} \leq C \rphi_{x,i}(h),
	\end{align}
	for all $q\in\N$. Using \eqref{eap2} and \textbf{A8},  the Davydov's inequality gives
	\begin{align*}
		J_{2,n}&= \sum_{S_2} \abs[\big]{\cov \parent[\big]{\Lambda_{i}^{(k,\ell )}(x),\Lambda_{j}^{(k,\ell )}(x)}}\leq C \rphi_{x}(h)^{2/q}\sum_{S_2} \alpha(\abs{i-j})^{1-2/q}\\
		&\leq  C \rphi_{x}(h)^{2/q}\sum_{i=1}^n\sum_{j:(i,j)\in S_2}\abs{i-j}^{-a(1-2/q)}\leq C \rphi_{x}(h)^{2/q}\sum_{i=1}^n\sum_{w=a_n+1}^\infty 2 w^{-a(1-2/q)},
	\end{align*}
	for all $q>2$. Note that 
	\begin{align*}
		\frac{a_n^{1-\rho}}{\rho-1}=\int_{a_n}^{\infty}t^{-\rho}dt=\sum_{w=a_n+1}^\infty \int_{w-1}^{w}t^{-\rho}dt\geq \sum_{w=a_n+1}^\infty \int_{w-1}^{w}w^{-\rho}dt=\sum_{w=a_n+1}^\infty w^{-\rho},
	\end{align*}
	for all $\rho>1$. Put $\rho=a(1-2/q)$. Then $\rho> 3(1-2/q)\geq 1$ if $q\geq 3$ and $a>3$. Thus, given that $a>3$ (\textbf{A8}),
	\begin{align*}
		J_{2,n}\leq C \rphi_{x}(h)^{2/q}n \frac{a_n^{1-a(1-2/q)}}{a(1-2/q)-1}=Cn\rphi_{x}(h),
	\end{align*}
	 by choosing $q\geq 3$ and 
 setting $a_n=\rphi_{x}(h)^{(1-2/q)/(1-a(1-2/q))}$ (which consists in a positive sequence diverging to $\infty$ for all $q\geq 3$). Additionally, $J_{1,n}$ is asymptotically dominated by $J_{2,n}$ for this choice of $a_n$  whenever $q$ is chosen large enough. To see this, let $f(q)\coloneqq\frac{1}{a-(1-2/q)^{-1}}$, given $a>3$.  Then $f$ is a monotone decreasing function on $[3,\infty)$ satisfying $\lim_{q\to\infty} f(q)=1/(a-1)$. By definition, $\forall \epsilon>0:\exists q_\epsilon:\forall q\geq q_\epsilon: f(q)-1/(a-1)<\epsilon$. Take $\epsilon=2p_\text{min}-1/(a-1)$. \textbf{A9} ensures\footnote{By hypothesis,  $1\geq 2p_\text{min}>1/2$ and $(a-1)^{-1}<1/2$.} that $\epsilon>0$ which in turn implies that $f(q)<2p_\text{min}$ as long as $q$ is chosen large enough. Thus, for all $q$ sufficiently large, $J_{1,n}=O\parent[\big]{n\rphi_{x}(h)\rphi_{x}(h)^{2p_\text{min}-f(q)}}=O(n\rphi_{x}(h)o(1))=o\parent[\big]{n\rphi_{x}(h)}$.

	Next, we focus on providing the lower bound for $S_{n,\ell ,k}^2$. Using the Law of Iterated Expectations, the conditional Jensen's inequality and Lemma 1(i)(iv), we have that 
	\begin{align*}
		S_{n,\ell ,k}^2&\geq V_n\geq n \min_{s\in[n]}\var\parent[\Big]{\Lambda_{s}^{(k,\ell )} (x)}\\
		&\geq \frac{n}{h^{2k}}\min_{s\in[n]}\chave[\Big]{E\parent[\big]{K_s^2\beta_s^{2k}\varphi_s^{2\ell}}-\colc[\big]{E\parent[\big]{K_s\abs{\beta_s^{k}\varphi_s^\ell}}}^2}\\
		&= \frac{n}{h^{2k}}\min_{s\in[n]}\chave[\Big]{E\parent[\big]{K_s^2\beta_s^{2k}E\parent[\big]{\abs{\varphi_s}^{2\ell}| \rchi_s}}-\colc[\big]{E\parent[\big]{K_s\abs{\beta_s}^{k}E\parent[\big]{\abs{\varphi_s}^\ell|\rchi_s}}}^2}\\
  &\geq \frac{n}{h^{2k}}\min_{s\in[n]}\chave[\Big]{E\parent[\big]{K_s^2\beta_s^{2k}\parent[\big]{E\parent[\big]{\abs{\varphi_s}^{\ell}| \rchi_s}}^2}-\colc[\big]{E\parent[\big]{K_s\abs{\beta_s}^{k}E\parent[\big]{\abs{\varphi_s}^\ell|\rchi_s}}}^2}\\
		&=\frac{n}{h^{2k}}\min_{s\in[n]}\chave[\Big]{\parent[\big]{E\parent[\big]{K_s^2\beta_s^{2k}}-\colc[\big]{E\parent[\big]{K_s\abs{\beta_s}^{k}}}^2}\underbrace{\parent[\big]{\gamma_{1 ,s}(x)+o(1)}^{2\ell}}_{\geq 0} \ }\\
		&\geq cn\min_{s\in[n]}\chave[\Big]{\parent[\big]{\rphi_{x,s}(h)-C\rphi_{x,s}(h)^2}\parent[\big]{\gamma_{1 ,s}(x)+o(1)}^{2\ell} }\\
          &= cn\min_{s\in[n]}\chave[\Big]{\rphi_{x,s}(h)(1-\underbrace{C\rphi_{x,s}(h)}_{=o(1)})\parent[\big]{\gamma_{1 ,s}(x)+o(1)}^{2\ell} }\\
		&\geq cn\min_{s\in[n]}\rphi_{x,s}(h)\min_{s\in[n]}\parent[\big]{\gamma_{1 ,s}(x)+o(1)}^{2\ell} \\
		&\geq cn\rphi_{x}(h)\min\parent[\Big]{1,\min_{s\in[n]}\parent[\big]{\gamma_{1 ,s}(x)+o(1)}^{2}}\geq cn\rphi_{x}(h),
	\end{align*} 
	for every $n$ sufficiently large. The proof is completed. \qed

\vspace{1em}

\noindent\textbf{Proof of Lemma 4}. The proof of Lemma 4 is similar to that of Lemma 1, and thus omitted.
\vspace{1em}

\noindent\textbf{Proof of Lemma 5}.
	Write
	\begin{align*}
		\sup_{x\in S}S_{n,\ell ,k}^2(x)&=\sup_{x\in S}\sum_{ S_1} \abs[\big]{\cov \parent[\big]{\Lambda_{i}^{(k,\ell )}(x),\Lambda_{j}^{(k,\ell )}(x)}}\\
		&+\sup_{x\in S}\sum_{S_2} \abs[\big]{\cov \parent[\big]{\Lambda_{i}^{(k,\ell )}(x),\Lambda_{j}^{(k,\ell )}(x)}}\\
		&+\sup_{x\in S}\sum_{i=1}^n \var(\Lambda_{i}^{(k,\ell )} )\coloneqq Q_{1,n}+Q_{2,n}+Q_{0,n},
	\end{align*}
	where $S_1=\{(i,j): 1\leq\abs{i-j}\leq a_n\}$ and  $S_2=\{(i,j): a_n+1\leq\abs{i-j}\leq n-1\}$  and $a_n$ is some positive real sequence diverging to infinity. Using similar arguments as those in the proof of  Lemma 3, we can obtain that $Q_{0,n}=O(n\rphi(h))$,
	\begin{equation}\label{equa7}
		Q_{1,n}=O\parent[\big]{na_n\rphi(h)^{1+1/2}},
	\end{equation}  and that 
	\begin{equation}\label{equa8}
		\sup_{x\in S}E\abs[\big]{\Lambda_{i}^{(k,\ell )}(x)}^q=O(\rphi(h)).
	\end{equation}
	In view of the mixing condition \textbf{H5}, the Davydov's inequality and \eqref{equa8} imply that,
	\begin{align*}
		Q_{2,n}&\leq C \rphi(h)^{2/q}\sum_{i=1}^n\sum_{w=a_n+1}^\infty  t^{(1-2/q)w}=C \rphi(h)^{2/q}\sum_{i=1}^n\sum_{w=a_n+1}^\infty  \int_{w-1}^w t^{(1-2/q)w}dz\\
		&\leq C \rphi(h)^{2/q}\sum_{i=1}^n\int_{a_n}^\infty  t^{(1-2/q)z}dz=C n\rphi(h)^{2/q} \frac{t^{(1-2/q)a_n}}{-(1-2/q)\ln t}\\
		&\leq C  n\rphi(h)^{2/q} t^{(1-2/q)a_n},
	\end{align*}
	for any $q>2$ and some $t\in(0,1)$. Set $a_n=\bar C d_n$, with $d_n$ being a positive sequence diverging to infinity and  $\bar C>0$ an arbitrary constant. Moreover, let $C^*>0$ be a any constant satisfying $-\bar C(1-2/q)(d_n/\ln d_n)\ln t\geq C^*$, for all sufficiently large $n$ (so that $\ln d_n>1$). Then
	\begin{align*}
		\ln t^{\bar C(1-2/q)d_n}\leq \ln d_n^{-C^*}\iff t^{\bar C(1-2/q)d_n}\leq  d_n^{-C^*}.
	\end{align*}
	As $d_n/\ln d_n\to\infty$, we can always  find $C^*>0$ that can be made arbitrarily large by taking $\bar C$ suitably large. From this observation, we immediately have that $Q_{2,n}=O(n\rphi_{x}(h)^{2/q} d_n^{-C^*})$. 
	
	Choose $d_n=\rphi(h)^{-1/2}$.
	Therefore, for any $q>2$ we can always pick a suitably large $C^*>0$ to obtain that 
	\begin{align*}
		Q_{1,n}&=O\parent[\big]{ n\rphi(h)^{1+1/2}d_n}= O\parent[\big]{n\rphi(h)},\\
		Q_{2,n}&=O\parent[\big]{n \rphi(h)^{2/q}d_n^{-C^*}}=o\parent[\big]{n\rphi(h)}.
	\end{align*} 
	Thus  $\sup_{x\in S}S_{n,\ell ,k}^2(x)=O(n\rphi_{x}(h))$. The lower bound for $\inf_{x\in S} S_{n,\ell ,k}^2(x)$ can be obtained in the same fashion as in the proof of Lemma 3. \qed

\vspace{1em}

\noindent\textbf{Proof of Lemma 6}.
	Write
	\begin{align*}
		\sup_{x\in S}W_{n,\ell }^2(x)&=\sup_{x\in S}\sum_{ S_1} \abs[\big]{\cov (M_{i}^{\ell }(x),M_{j}^{\ell }(x))}+\sup_{x\in S}\sum_{S_2} \abs[\big]{\cov (M_{i}^{\ell }(x),M_{j}^{\ell }(x))}\\
		&+\sup_{x\in S}\sum_{i=1}^n \var(M_{i}^{\ell }(x) )
		\coloneqq G_{1,n}+G_{2,n}+G_{3,n},
	\end{align*}
	where the sets $S_1$ and $S_2$ are defined exactly in the same manner as in the proof of Lemma 3. For $x\in S$, define the index $j(x)=\argmin_{j\in[N_{r_n}(S)]}d(x,x_j)$ and denote  $B_{x,n}\coloneqq B(x,h)\cup B(x_{j(x)},h)$. Moreover, given a measurable set $E$ and a random variable $\rchi$, define $1_A(\rchi)=1$. Then, by \textbf{H7}(i),
	\begin{align*}
		G_{1,n}&=\parent[\bigg]{\frac{r_n}{h}}^2 \sup_{x\in S}\sum_{S_1}\abs[\Big]{E\colc[\big]{\abs{\varphi_i\varphi_j}^\ell 1_{B_{x,n}}(\rchi_i)1_{B_{x,n}}(\rchi_j)}\\ 
			&-E\colc[\big]{\abs{\varphi_i}^\ell 1_{B_{x,n}}(\rchi_i)}E\colc[\big]{\abs{\varphi_j}^\ell 1_{B_{x,n}}(\rchi_j)}}\\ 
		&\leq C\parent[\bigg]{\frac{r_n}{h}}^2\sup_{x\in S}\sum_{S_1}\chave[\big]{ E\colc[\big]{1_{B_{x,n}}(\rchi_i)1_{B_{x,n}}(\rchi_j)}+E\colc[\big]{1_{B_{x,n}}(\rchi_i)}E\colc[\big]{1_{B_{x,n}}(\rchi_j)}}.
	\end{align*}
	In general, for any measurable sets $A$ and $B$, we have that $1_{A\cup B}(\rchi)\leq 1_A(\rchi)+1_B(\rchi)$. Therefore, $$1_{B_{x,n}}(\rchi_i)1_{B_{x,n}}(\rchi_j)\leq [1_{B(x,h)}(\rchi_i)+1_{B(x_{j(x)},h)}(\rchi_i)][1_{B(x,h)}(\rchi_j)+1_{B(x_{j(x)},h)}(\rchi_j)],$$ which implies that
	\begin{align*}
		E\colc[\big]{1_{B_{x,n}}(\rchi_i)1_{B_{x,n}}(\rchi_j)}&\leq \Psi_{x,i,j}(h)+\Psi_{x_{j(x)},i,j}(h)+\Psi_{x,x_{j(x)},i,j}(h)+\Psi_{x_{j(x)},x,i,j}(h)\\
		&\leq 4C \sup_{x\in S}\max_{s\in [n]}\rphi_{x,s}(h)^{1+2p_{\min}}\leq C \rphi(h)^{1+2p_{\min}},
	\end{align*}
	using \textbf{H1} and \textbf{H6}, for all $n$ large enough, where $p_{\min}=\min_{(i,j)\in [n]^2} p_{1,i,j}$. Similarly,
	\begin{align*}
		E\colc[\big]{1_{B_{x,n}}(\rchi_i)}E\colc[\big]{1_{B_{x,n}}(\rchi_j)}&\leq \parent[\big]{\rphi_{x,i}(h)+\rphi_{x_j(x),i}(h)}\parent[\big]{\rphi_{x,j}(h)+\rphi_{x_j(x),j}(h)}\\
		&\leq C \rphi(h)^2,
	\end{align*}
	uniformly on $x\in S$. Thus, for $n$ sufficiently large,
	\begin{align}\notag
		G_{1,n}\leq C \parent[\bigg]{\frac{r_n}{h}}^2\sum_{S_1}\chave[\big]{\rphi(h)^{1+2p_{\min}}+\rphi(h)^2}\leq C \parent[\bigg]{\frac{r_n}{h}}^2 na_n\rphi(h)^{1+2p_{\min}}.
	\end{align}
	Next, we bound $G_{2,n}$ using the Davydov's inequality. For this, note that for any $q>2$,
	\begin{align*}
		\sup_{x\in S} E\abs{Z_i^\ell (x)}^q&=\parent[\bigg]{\frac{r_n}{h}}^q\sup_{x\in S}E\abs[\big]{\chave[\big]{\abs{\varphi_i}^\ell 1_{B_{x,n}}(\rchi_i) - E\parent[\big]{\abs{\varphi_i}^\ell 1_{B_{x,n}}(\rchi_i)} }^q}\\
		&=\parent[\bigg]{\frac{r_n}{h}}^q\sup_{x\in S} E\abs[\bigg]{\sum_{k=0}^q \binom{q}{k}\chave{\abs{\varphi_i}^{\ell k}1_{B_{x,n}}(\rchi_i)}\chave{E(1_{B_{x,n}}(\rchi_i)\gamma_{\ell ,i}(\rchi_i))}^{q-k} }\\
		&\leq C\parent[\bigg]{\frac{r_n}{h}}^q\sum_{k=0}^q \sup_{x\in S} E(1_{B_{x,n}}(\rchi_i))\chave{E(1_{B_{x,n}}(\rchi_i))}^{q-k}\\
		&\leq C\parent[\bigg]{\frac{r_n}{h}}^q\sum_{k=0}^q\rphi(h)^{1+q-k}\leq  C\parent[\bigg]{\frac{r_n}{h}}^q\rphi(h).
	\end{align*}
	Hence, for all $q>2$ and some $t\in(0,1)$, by the Davydov's inequality and \textbf{H5}, it follows that
	\begin{align*}
		G_{2,n}\leq C\parent[\bigg]{\frac{r_n}{h}}^2\rphi(h)^{2/q}\sum_{S_2}\alpha(\abs{i-j})^{1-2/q}\leq C\parent[\bigg]{\frac{r_n}{h}}^2\rphi(h)^{2/q}t^{(1-2/q)a_n} .
	\end{align*}
	Along the lines of the proof of Lemma 5, we set $a_n=C_1\rphi(h)^{-1/2}$ to obtain that $G_{1,n}=O\parent[\big]{(r_n/h)^2n\rphi(h)}$ and that $	G_{2,n}=O\parent[\big]{(r_n/h)^2n\rphi(h)\colc[\big]{\rphi(h)^{2/q-1+C'/2}}}=o\parent[\big]{(r_n/h)^2n\rphi(h)}$ for some suitably large $C'>0$.
	
	The last term $G_{3,n}$ can be bounded as follows
	\begin{align*}
		G_{3,n}&\leq C\parent[\bigg]{\frac{r_n}{h}}^2\sum_{i=1}^n\chave[\bigg]{ \sup_{x\in S}E\parent[\big]{1_{B(x,h)}(\rchi_i)+1_{B(x_{j(x)},h)}(\rchi_i)}\\
			&+\sup_{x\in S}\colc[\big]{E\parent[\big]{1_{B(x,h)}(\rchi_i)+1_{B(x_{j(x)},h)}(\rchi_i) }}^2}\\
		&\leq C\parent[\bigg]{\frac{r_n}{h}}^2n\chave[\big]{\rphi(h)+\rphi(h)^2}\leq C\parent[\bigg]{\frac{r_n}{h}}^2n\rphi(h) .
	\end{align*}
	Thus, $\sup_{x\in S}W_{n,\ell }^2(x)=O\parent[\big]{(r_n/h)^2n\rphi(h)}$.
	Now, we focus on the uniform lower bound. It holds that
	\begin{align*}
		\inf_{x\in S} W_{n,\ell }^2(x)&\geq \inf_{x\in S} \sum_{1}^{n}\var \parent[\big]{Z_i^\ell (x)}\\
		&\geq cn \parent[\bigg]{\frac{r_n}{h}}^2\inf_{x\in S}\chave[\Big]{
			\min_{s\in[n]}\colc[\big]{
				E\parent[\big]{1_{B_{x,n}}(\rchi_i) }  -\parent[\big]{E\parent[\big]{1_{B_{x,n}}(\rchi_i) }}^2
			}
		}\\
		&\geq c n \parent[\bigg]{\frac{r_n}{h}}^2\chave[\bigg]{\inf_{x\in S}\min_{s\in[n]}E\parent[\big]{1_{B_{x,n}}(\rchi_i) } -\sup_{x\in S}\max_{s\in[n]}\parent[\big]{E\parent[\big]{1_{B_{x,n}}(\rchi_i) }}^2 }.
	\end{align*}
	In general, for measurable sets $A$ and $B$, we have that
	$1_{A\cup B}(\rchi)=1_A(\rchi)+1_B(\rchi)-1_{A\cap B}(\rchi)\geq 1_A(\rchi)+1_B(\rchi)-1_{B}(\rchi)=1_A(\rchi)$. Then 
	\begin{equation*}
		\inf_{x\in S} W_{n,\ell }^2(x)\geq  c n \parent[\bigg]{\frac{r_n}{h}}^2\colc[\big]{\rphi(h)-C\rphi(h)^2}\geq  c n \parent[\bigg]{\frac{r_n}{h}}^2\rphi(h).
	\end{equation*}
Finally, 
\begin{align*}
	\sup_{x\in S} \max_{i\in[n]}ET^\ell _i(x)&\leq C\frac{r_n}{h}\sup_{x\in S} \max_{i\in[n]}E\parent[\big]{1_{B(x,h)}(\rchi_i)+1_{B(x_{j(x)},h)}(\rchi_i)}\\
	&\leq C\frac{r_n}{h}\rphi(h).
\end{align*}
 \qed

\vspace{1em}

\section{Theoretical details} \label{sec:3}
Our framework considers a collection of random pairs $(Y_i,\rchi_i)$, $i\in\{1,\dotsc,n\}$,  on $(\Omega,\A,P)$  taking values in $\R\times \F$, where $\F$ is some semimetric space equipped with a semimetric (sometimes called by a pseudometric) $d$. Throughout the paper, $d(x,\rchi_i)$ is implicitly assumed to be a random variable for any $i\in\{1,\dotsc,n\}$ and some point $x\in\F$. The following proposition clarifies the measurabilty of $d(x,\rchi_i)$.

\begin{prop}\label{prop1}
	Let $x\in\F$ and $\rchi_i:(\Omega,\A)\to (\F,\A')$, where $\A'=\B(\F)$ is the sigma-algebra generated by the topology on $(\F,d)$ induced by $d$. Then $d(x,\rchi_i)$ is a random variable.
\end{prop}
\noindent\textbf{Proof of Proposition \ref{prop1}}
	Let $\tau_d$ denote the topology on the semimetric space $(\F,d)$ induced by $d$ and $B_r(x)$ denote the open ball of radius $r$ and center $x$ with respect to $d$. Then it can be shown that $\{B_{1/j}(x)\}_{j\in\N*}$ is a local basis for  $\tau_d$, which in turn implies that $(\F,\tau_d)$ is first countable.
	
	Let $y\in(\F,d)$ and $\{y_j\}\subseteq\F$ be a sequence such that $y_j\to y$ as $j\to\infty$. Let $\epsilon>0$ be arbitrary. Then $\exists j_0:\forall j>j_0:d(y,y_j)<\epsilon$. Hence, we have that for $j>j_0$
	\begin{align}\label{e1}
		d(y_j,x)&\leq d(y_j,y)+d(y,x)<d(y,x)+\epsilon, \text{ and}\\\label{e2}
		d(x,y)&\leq d(x,y_j)+d(y_j,y)<d(x,y_j)+\epsilon.
	\end{align}
	The combination of \eqref{e1} and \eqref{e2} implies that $\abs{d(x,y)-d(x,y_j)}<\epsilon$. Since $\epsilon>0$ is arbitrary, $d(x,y_j)\to d(x,y), j\to\infty,$ and so the function $d(\cdot,x)$ is sequentially continuous.
	
	It is known that \citep[Theorem 21.3 of][]{munkres} a sequentially continuous function $d(\cdot,x):\F\to\R$ with a first countable domain is continuous.
	
	Next, consider the set $\mathcal W=\{A\subseteq\R: d(\cdot,x)^{-1}[A]\in\B(\F)\}$. It is easily checked that $\mathcal W$ is a sigma-algebra. By the continuity of $d(\cdot,x)$, $\mathcal W$ contains all the open sets in $\R$, that is, $\tau\subseteq \mathcal W$ where $\tau$ is the standard topology on $\R$. By the minimality of generated sigma-algebras, $\sigma(\tau)=\B(\R)\subseteq \mathcal W$, which means that the inverse image of any Borel set in $\R$ is also in $\B(\F)$. Thus $d(\cdot,x)$ is $\B(\F)-\B(\R)$ measurable. Therefore the composition of $d(\cdot,x): (\F,\B(\F))\to(\R,\B(\R))$ and $\rchi_i:(\Omega,\A)\to(\F,\B(\F))$ is $\A-\B(\R)$ measurable, as desired.
\qed

In general, the strong mixing dependence of $\{Y_i,\rchi_i\}$ affects the treatment of the associated sum of covariances, variance of partial sums and exponential inequalities. In our work, the Davydov's  and the Fuk-Nagaev's inequalities are used to bound covariances and probabilities of sums, respectively. More precisely, the application of such inequalities are possible in our proofs as long as the sequence $\{d(x,\rchi_i)\}_{i\in\N*}$ inherits the assumed strong mixing properties of $\{Y_i,\rchi_i\}_{i\in\N^*}$. The next result shows the relation between the mixing dependence of $\{(Y_i,\rchi_i)\}$ and the  mixing dependence of $\{d(x,\rchi_i)\}$.

\begin{prop}\label{prop2}
	Let $x\in\F$. If $\{(Y_i,\rchi_i)\}$ is strongly mixing with mixing coefficient $\alpha$ and $(\F,\B(\F))$ is a measurable space, then the sequence $\{d(x,\rchi_i)\}$ is also strongly mixing with mixing coefficient $\alpha$.
\end{prop}
\noindent\textbf{Proof of Proposition \ref{prop2}}
	For any natural numbers $0<w\leq k$, it holds that
	\begin{align*}
		\mathcal F_w^k&\coloneqq \sigma((Y_i,\rchi_i),w\leq i\leq k)=\sigma\parent[\bigg]{\bigcup\limits_{i=w}^k \sigma((Y_i,\rchi_i))}=\sigma\parent[\bigg]{\colc[\bigg]{\bigcup\limits_{i=w}^k \sigma(Y_i)}\cup \colc[\bigg]{\bigcup\limits_{i=w}^k \sigma(\rchi_i)}}\\&
		\supseteq \sigma\parent[\bigg]{\bigcup\limits_{i=w}^k \sigma(\rchi_i)}=\sigma(\rchi_i,w\leq i\leq k)\coloneqq \mathcal G_w^k.
	\end{align*}
	By Proposition \ref{prop1}, $\{d(x,\rchi_i)\}$  is a sequence of random variables. Clearly, $\sigma(d(x,\rchi_i))=\chave[\big]{\rchi_i^{-1}\circ d(\cdot,x)^{-1}[A]:A\in\B(\R)}\subseteq \chave[\big]{\rchi_i^{-1}[B]:B\in\B(\F)}=\sigma(\rchi_i)$, and hence, $\mathcal G_w^k\supseteq \sigma(d(x,\rchi_i),w\leq i\leq k)\coloneqq \mathcal D_w^k$. From the transitivity of subsets, $\mathcal F_w^k\supseteq \mathcal D_w^k$. Thus
	\begin{equation*}
		\sup_{k\in\N} \chave[\big]{\abs{P(A)P(B)-P(A\cap B)}:A\in \mathcal D_1^k, B\in\mathcal D_{k+j}^\infty}\leq \alpha(j).
	\end{equation*}
	\qed

In the time series context, the notion of mixing sequences is connected to the asymptotic independence of random variables. A mixing time series can be regarded as a sequence of random variables for which the past and the future is asymptotically independent. In assuming that $\{(Y_i,\rchi_i)\}$ is strongly mixing, we are indirectly restricting the relation between  $\Psi_{x,i,j}(r_1,r_2,r_3,r_4)$ and $\rphi_{x,i}(r_1,r_2)\rphi_{x,j}(r_3,r_4)$, whenever $\abs{i-j}$ is relatively \textcolor{violet}{``large''}, where for an easy reference
\begin{align*}
	&\rphi_{x,i}(r_1,r_2)=P(r_1\leq d(x,\rchi_i)\leq r_2),\\
	&\Psi_{x,i,j}(r_1,r_2,r_3,r_4)=P(r_1\leq d(x,\rchi_i)\leq r_2,r_3\leq d(x,\rchi_j)\leq r_4),
\end{align*}
with the notations $\rphi_{x,i}(r_2)\coloneqq \rphi_{x,i}(0,r_2)$ and $\Psi_{x,i,j}(r_2)\coloneqq \Psi_{x,i,j}(0,r_2,0,r_2)$ for any $r_1,r_2,r_3,r_4>0$.

 In view of the work of \cite{leulmi}, it is of interest to ask whether it is compatible to assume that $\{(Y_i,\rchi_i)\}$ is strongly mixing and $\Psi_{x,i,j}(h)=\Theta\parent[\big]{\colc{\phi_{x,i}(h)\phi_{x,j}(h)}^{1/2+p}}$ for some $p\neq 1/2$ and all $i,j$. 

\begin{prop}\label{prop33}
	Let $x\in\F$. If $\{(Y_i,\rchi_i)\}_{i\in\N^*}$ is strongly mixing, then 
	\begin{equation}\label{p33}
		\Psi_{x,i,j}(h)=\Theta\parent[\big]{\phi_{x,i}(h)\phi_{x,j}(h)}, n\to\infty,
	\end{equation}
	for any $i,j\in[n]$ such that $\abs{i-j}\geq b_n$ and $\alpha(b_n)=o \parent[\big]{ \phi_{x,i}(h)\phi_{x,j}(h)}$.  In particular, if $\{(Y_i,\rchi_i)\}_{i\in\N^*}$ is strongly mixing with arithmetic rate $a>3$ or with geometric rate, then \textup{\eqref{p33}} holds for $b_n= \colc[\big]{\min_{s\in[n]}\rphi_{x,s}(h)}^{-1}$.
\end{prop}
\noindent\textbf{Proof of Proposition \ref{prop33}}
	Let $A=\{0\leq d(x,\rchi_i)\leq h\}$ and $B=\{0\leq d(x,\rchi_j)\leq h\}$. Clearly $A\in\sigma(d(x,\rchi_i))$ because $A=d(x,\rchi_i)^{-1}([0,h])$, and by the same argument, $B\in\sigma(d(x,\rchi_j))$. From Proposition \ref{prop2}, it immediately follows that $\alpha(\abs{i-j})\geq \abs{P(A\cap B)-P(A)P(B)}$, and so,
	\begin{align*}
		P(A\cap B)&\leq  P(A)P(B)+\alpha(\abs{i-j}), \text{ and }\\
		P(A\cap B)&\geq  P(A)P(B)-\alpha(\abs{i-j}).
	\end{align*}
	
	Since $\{\alpha(k)\}$ is a nonincreasing sequence and $b_n$ is such that $\alpha(b_n)=o \parent[\big]{ \phi_{x,i}(h)\phi_{x,j}(h)}$, for any $\abs{i-j}>b_n$ 
	\begin{align*}
		\Psi_{x,i,j}(h)&\geq \rphi_{x,i}(h)\rphi_{x,j}(h)-\alpha(b_n)=\rphi_{x,i}(h)\rphi_{x,j}(h)\colc[\bigg]{1-\frac{\alpha(b_n)}{\rphi_{x,i}(h)\rphi_{x,j}(h)}}\\
		&\geq c \rphi_{x,i}(h)\rphi_{x,j}(h),
	\end{align*}
	whenever $n$ is sufficiently large. Likewise, for any $\abs{i-j}>b_n$ and $n$ large enough
	\begin{equation*}
		\Psi_{x,i,j}(h)\leq \rphi_{x,i}(h)\rphi_{x,j}(h)+\alpha(b_n)\leq C \rphi_{x,i}(h)\rphi_{x,j}(h).
	\end{equation*}
	
	For the second result, it suffices to show that $\alpha\parent[\big]{\colc[\big]{\min_{s\in[n]}\rphi_{x,s}(h)}^{-1}}=o(\rphi_{x,i}(h)\rphi_{x,j}(h))$. Suppose that $\{(Y_i,\rchi_i)\}_{i\in\N^*}$ is strongly mixing with arithmetic rate $a>3$. Then
	\begin{align*}
		\frac{\alpha\parent[\big]{\colc[\big]{\min_{s\in[n]}\rphi_{x,s}(h)}^{-1}}}{\rphi_{x,i}(h)\rphi_{x,j}(h)}\leq C \frac{\min_{s\in[n]}\rphi_{x,s}(h)^a}{\rphi_{x,i}(h)\rphi_{x,j}(h)}\leq C \min_{s\in[n]}\rphi_{x,s}(h)^{a-2}=o(1).
	\end{align*}
	Now, assume the geometric mixing rate: $\alpha(k)\leq Ct^k, t\in(0,1)$. It holds that
	\begin{align*}
		\frac{\alpha\parent[\big]{\colc[\big]{\min_{s\in[n]}\rphi_{x,s}(h)}^{-1}}}{\rphi_{x,i}(h)\rphi_{x,j}(h)}\leq C \frac{t^{1/\min_{s\in[n]}\rphi_{x,s}(h)}}{\rphi_{x,i}(h)\rphi_{x,j}(h)}\leq C \frac{t^{1/\min_{s\in[n]}\rphi_{x,s}(h)}}{(\min_{s\in[n]}\rphi_{x,s}(h))^2}=o(1),
	\end{align*}
	since the geometric decay rate dominates any polynomial rate of decay.
\qed

In the context of equally distributed and mixing data, note that if it is assumed that $\ln n/(n\rphi_{x,1}(h))=o(1)$, then there will always be indices $i,j\in[n]: \abs{i-j}\geq 1/\rphi_{x,1}(h)$ for suitably large $n$. For these indices, we cannot assume $C'\rphi_{x,1}(h)^{1/2+p} \leq \Psi_{x,i,j}(h)\leq C\phi_{x,1}(h)^{1/2+p}$, for some $p\neq 1/2$ and any large $n$, since Proposition \ref{prop3} shows that $\Psi_{x,i,j}(h)=\Theta\parent[\big]{\rphi_{x,1}(h)^2}$.

It is known that the concentration properties of $\rchi_i$ affect the convergence rates of kernel estimators \citep[see Chapter 13 of][]{ferraty}. In particular, processes of fractal order showed good convergence rates when compared to other types of processes (e.g. exponential-type). The generality of a fractal-type process can be seen from Lemma 13.6 of \cite{ferraty}, where it is shown  that we can construct a semimetric $d$ for which any process, valued in a separable Hilbert space with inner product, is necessarily of fractal-type.

For the next proposition, we recall some notations of Section 3 of the paper. Let $x$ be an element in $\F$, $h>0$ be a bandwidth such that $h=o(1)$ and define $\rphi_{x,i}(r_1,r_2)=P(r_1\leq d(x,\rchi_i)\leq r_2)$ with $\rphi_{x,i}(r_2)\coloneqq P(0\leq d(x,\rchi_i)\leq r_2)$ for any $r_1,r_2>0$. The generic positive small and large constants, respectively, $c$ and $C$, are also used from now on.

\begin{prop}\label{prop3}
	Suppose that $\rchi_i$ is of fractal order\footnote{We say that a random variable $\rchi_i$ is of fractal order $\theta$ with respect to a semimetric $d$ if $\rphi_{x,i}(\epsilon)\overset{a}{\approx}C_x\epsilon^\theta, \epsilon\to 0$, for some constant $C_x>0$.} $\theta>0$, $L>0$ and that the kernel function is defined by
	\begin{itemize}
		\item $K(t)=L^{-(\alpha+1)}(1+1/\alpha)(L^\alpha-t^\alpha)1_{[0,L]}(t), \ \alpha>0$; or
		\item $K(t)=1_{[0,L]}(t)/L.$
	\end{itemize}
	Then for $n$ sufficiently large, it holds that 
	\begin{align}\label{e3}
		&\int_0^\epsilon \rphi_{x,i}(zh,\epsilon h)\frac{d}{dz}(z^\ell K(z))dz>c\rphi_{x,i}(h), \ 0<\epsilon\leq L, \ \ell \in\N^*; \\\label{e4}
		&\int_0^h \rphi_{x,i}(z)^\delta dz>ch\rphi_{x,i}(h)^\delta, \text{ for any } \delta>0;
		\\\label{ee6}
		&\int_0^1  \rphi_{x,i}(zh, h)^{1/2+p}\frac{d}{dz}(z^2K(z))dz>c\rphi_{x,i}(h)^{1/2+p},\text{ for any }0<p<1, \text{ and }L=1.
	\end{align} 
In addition, if Assumption A9 (Section 3.2) is adapted to 
\vspace{1em}
\begin{enumerate}
\item [\textbf{A9'}.] $\forall i,j\in[n]: \forall w,t\in[0,1]: \exists 1/4<p_{1,i,j}\leq 1/2\leq p_{2,i,j}< 3/4$ such that
\begin{equation*}
	[\rphi_{x,i}(th,h)\rphi_{x,j}(wh,h)]^{1/2+p_{2,i,j}}\leq \Psi_{x,i,j}(th,h,wh,h)\leq  [\rphi_{x,i}(th,h)\rphi_{x,j}(wh,h)]^{1/2+p_{1,i,j}},
\end{equation*}
\end{enumerate}
then Assumption A10 holds.
\end{prop}
\noindent\textbf{Proof of Proposition \ref{prop3}}
	For simplicity, put $k_0=L^{-(\alpha+1)}(1+1/\alpha)>0$. For the polynomial-type kernel, we have that on $[0,L]$
	\begin{align*}
		\rphi_{x,i}(zh,\epsilon h)\frac{d}{dz}(z^\ell K(z))&=C_xh^\theta k_0\colc[\big]{\ell L ^\alpha\epsilon^\theta z^{\ell -1}-(\ell +\alpha)\epsilon^\theta z^{\ell +\alpha-1}\\
			&-\ell L^\alpha z^{\ell +\theta-1}+(\ell +\alpha)z^{\ell +\alpha+\theta-1}} +o(h^\theta),
	\end{align*}
	and so,
	\begin{align*}
		\int_{0}^{\epsilon}\rphi_{x,i}(zh,\epsilon h)\frac{d}{dz}(z^\ell K(z))dz&=C_xh^\theta k_0\epsilon^{\ell +\theta}\colc[\bigg]{L^\alpha-\epsilon^\alpha-\frac{\ell L^\alpha}{\ell +\theta}+\frac{\ell +\alpha}{\ell +\alpha+\theta}\epsilon^\alpha}\\
		&+o(h^\theta)\\
		&=C_xh^\theta k_0 \epsilon^{\ell +\theta}\theta \chave[\bigg]{\frac{\alpha L^\alpha+(\ell +\theta)(L^\alpha-\epsilon^\alpha)}{(\ell +\theta)(\ell +\theta+\alpha)}}+o(h^\theta)\\
		&\coloneqq C_x h^\theta M+o(h^\theta).
	\end{align*}
	Since $L\geq \epsilon$, $M>0$. Then
	\begin{equation*}
		\frac{1}{\rphi_{x,i}( h)}\int_{0}^{\epsilon}\rphi_{x,i}(zh,\epsilon h)\frac{d}{dz}(z^\ell K(z))dz\to M, n\to\infty,
	\end{equation*}
	and thus, there is $c>0$ such that \eqref{e3} holds for $n$ large enough. Similarly, if $K$ is the uniform kernel, then it can be shown that
	\begin{equation*}
		\int_{0}^{\epsilon}\rphi_{x,i}(zh,\epsilon h)\frac{d}{dz}(z^\ell K(z))dz=C_xh^\theta \frac{\theta}{\ell +\theta}\frac{\epsilon^{\ell +\theta}}{L}+o(h^\theta)
	\end{equation*}
	implying \eqref{e3} for sufficiently large $n$.
	
	Next,
	\begin{equation*}
		\int_{0}^{h}\rphi_{x,i}(z)^\delta dz=C_x\frac{h^{\delta\theta+1}}{1+\delta\theta}+o(h^{\delta\theta+1})
	\end{equation*}
	which implies \eqref{e4} for all $n$ large enough.

	Next, we focus on the inequality \eqref{ee6}. Suppose the kernel funtion is of polynomial-type. By applying the integration by parts technique, we obtain that
	\begin{align}\notag
		\int_0^1  \rphi_{x,i}(zh, h)^{1/2+p}\frac{d}{dz}&(z^2K(z))dz\\\notag
		\overset{a}{\approx}& \ C_x\frac{\alpha+1}{\alpha}h^{\theta(1/2+p)}\int_{0}^{1}\parent[\big]{1-z^\theta}^{1/2+p}\colc[\big]{2z-(2+\alpha)z^{\alpha+1}}dz\\\notag
		=&\ C_x\frac{\alpha+1}{\alpha}h^{\theta(1/2+p)}\chave[\bigg]{\parent[\big]{1-z^\theta}^{1/2+p}\parent[\big]{z^2-z^{2+\alpha}}\Big|_0^1\\\notag
        &\qquad+\int_0^1\parent[\bigg]{\frac{1}{2}+p}\parent[\big]{1-z^{\theta}}^{p-1/2}\theta z^{\theta-1}\parent[\big]{z^2-z^{\alpha+2}}dz}\\\label{ee8}
		=&\ C_x\frac{\alpha+1}{\alpha}h^{\theta(1/2+p)}\parent[\bigg]{\frac{1}{2}+p}\theta \int_0^1\parent[\big]{1-z^{\theta}}^{p-1/2}z^{\theta-1}\parent[\big]{z^2-z^{\alpha+2}}dz.
	\end{align}
	
	Let $g(z)=z^2-z^{\alpha+2}=z^2\parent{1-z^\alpha}\geq0$. We can check that $[2/(2+\alpha)]^{1/\alpha}\in\parent[\big]{1/\sqrt{e},1}$ is the unique critical point of $g$ on $(0,1)$ and $g$ is increasing on $\parent[\big]{[1/(2+\alpha)]^{1/\alpha},[2/(2+\alpha)]^{1/\alpha}}$. Therefore, by applying the substitution $w=1-z^\theta$, 
	\begin{align*}
		&\parent[\bigg]{\frac{1}{2}+p}\theta \int_0^1\parent[\big]{1-z^{\theta}}^{p-1/2}z^{\theta-1}g(z)dz\\
		&\ > \parent[\bigg]{\frac{1}{2}+p}\theta g\parent[\big]{[1/(2+\alpha)]^{1/\alpha}} \int_{[1/(2+\alpha)]^{1/\alpha}}^{[2/(2+\alpha)]^{1/\alpha}}\parent[\big]{1-z^{\theta}}^{p-1/2}z^{\theta-1}dz\\
		&\ = \parent[\bigg]{\frac{1}{2}+p} g\parent[\big]{[1/(2+\alpha)]^{1/\alpha}}\int_{1-[1/(2+\alpha)]^{\theta/\alpha}}^{1-[2/(2+\alpha)]^{\theta/\alpha}}-w^{p-1/2}dw\\
		&\ =\chave[\bigg]{\parent[\bigg]{\frac{1}{2+\alpha}}^{2/\alpha}-\parent[\bigg]{\frac{1}{2+\alpha}}^{(2+\alpha)/\alpha}}
		\chave[\bigg]{\colc[\bigg]{1-\parent[\bigg]{\frac{1}{2+\alpha}}^{\theta/\alpha}}^{1/2+p}-\colc[\bigg]{1-\parent[\bigg]{\frac{2}{2+\alpha}}^{\theta/\alpha}}^{1/2+p}}\\
		&\ \coloneqq M\frac{\alpha}{\alpha+1}>0.
	\end{align*}
	
	Hence, 
	\begin{equation*}
		L\coloneqq \lim_{n\to\infty}\frac{1}{\rphi_{x,i}(h)^{1/2+p}}\int_0^1  \rphi_{x,i}(zh, h)^{1/2+p}\frac{d}{dz}(z^2K(z))dz>M>0,
	\end{equation*}
	which implies \eqref{ee6}, for all $n$ sufficiently large.
	
	To show the last result, note that there must have some $p\in(1/4,3/4)$ such that $$[\rphi_{x,i}(th,h)\rphi_{x,j}(wh,h)]^{1/2+p}= \Psi_{x,i,j}(th,h,wh,h),$$
	for any $i,j\in[n], w,t\in[0,1]$ if A9' holds. Therefore, by \eqref{ee6},
	\begin{align*}
		\int_0^{1} \Psi_{x,i,j}(zh,h,0,h)\frac{d}{dz}(z^2K(z))dz&=\int_0^{1} [\rphi_{x,i}(zh,h)\rphi_{x,j}(h)]^{1/2+p}\frac{d}{dz}(z^2K(z))dz\\
		&>c[\rphi_{x,i}(h)\rphi_{x,j}(h)]^{1/2+p}=c\Psi_{x,i,j}(h).
	\end{align*}
 If $K$ is of type (II), then by \eqref{ee6} and \eqref{e4}, we have that
 \begin{align*}
 	-\iint_0^{1}& \Psi_{x,i,j}(zh,h,0,wh)\frac{d}{dz}(z^2K(z))K'(w)dzdw\\
 	&=-\iint_0^{1} [\rphi_{x,i}(zh,h)\rphi_{x,j}(wh)]^{1/2+p}\frac{d}{dz}(z^2K(z))K'(w)dzdw\\
 	&>c\int_0^1 [\rphi_{x,i}(zh,h)]^{1/2+p}\frac{d}{dz}(z^2K(z))dz \int_0^1
 	[\rphi_{x,j}(wh)]^{1/2+p} dw\\
 	&>c [\rphi_{x,i}(h)]^{1/2+p}h^{-1} \int_0^h
 	[\rphi_{x,j}(t)]^{1/2+p} dt\\
 	&>c [\rphi_{x,i}(h)\rphi_{x,j}(h)]^{1/2+p}=c\Psi_{x,i,j}(h).
 \end{align*}
\qed

Proposition \ref{prop3} shows that the hypotheses \textbf{A6}(i) and \textbf{A10} of the paper hold for general fractal-type processes when popular kernel functions (such as the triangle, quadratic, cubic or uniform asymmetric kernels) are chosen and \textbf{A9} is strengthened appropriately. The condition \textbf{A9} is obtained from  \textbf{A9'} if in particular we set $w=t=0$. 

In what follows, we will present the basic properties of $\Oac$ and $\oac$.
\begin{prop}\label{propdan}
	Let $\{a_n\}$ and $\{b_n\}$ be two real sequences and let $X_n=\Oac(a_n)$. Then
	\begin{enumerate}
		\item[(i)] $Z_n=\Oac(b_n) \text{ or } Z_n=O(b_n)\implies X_n Z_n=\Oac(a_nb_n)$;
		\item[(ii)] $Z_n=\oac(b_n) \text{ or } Z_n=o(b_n)\implies X_n Z_n=\oac(a_nb_n)$;
		\item[(iii)] $Z_n=\Oac(b_n) \text{ or } Z_n=O(b_n) \implies X_n +Z_n=\Oac(a_n+b_n)$;
		\item[(iv)] $Z_n=\oac(b_n) \text{ or } Z_n=o(b_n) \implies X_n +Z_n=\Oac(a_n+b_n)$;
		\item[(v)]  $a_n=O(b_n)\implies X_n=\Oac(b_n)$;
		\item[(vi)] $a_n=o(b_n)\implies X_n=\oac(b_n)$;
		\item[(vii)] $\Oac(a_n)=a_n\Oac(1)$.
	\end{enumerate}
\end{prop}
\noindent\textbf{Proof of Proposition \ref{propdan}}
	Denote by $\epsilon_x>0$ the number such that $\sum_{n\in\N}P(\abs{X_n/a_n}>\epsilon_x)<\infty$.
	
	(i) If $Z_n=\Oac(b_n)$, then $\exists \epsilon_z>0$ such that $\sum_{n\in\N}\{P(\abs{X_n/a_n}>\epsilon_x)+P(\abs{Z_n/b_n}>\epsilon_z)\}<\infty$. By taking $\epsilon=\epsilon_x\epsilon_z$, we get
	\begin{align*}
		\sum_{n\in\N}P(\abs{X_nZ_n}>\epsilon \abs{a_nb_n})&\leq 	\sum_{n\in\N}\chave{P(\abs{X_nZ_n}>\epsilon \abs{a_nb_n},\abs{X_n}\leq\epsilon_x\abs{a_n})+P(\abs{X_n}>\epsilon_x\abs{a_n})}\\
		&\leq \sum_{n\in\N}\chave{P(\abs{Z_n}>\epsilon_z \abs{b_n})+P(\abs{X_n}>\epsilon_x\abs{a_n})}<\infty.
	\end{align*}
	If $Z_n=O(b_n)$, then $\exists M>0,n_0\in\N:\forall n>n_0:\abs{Z_n/b_n}\leq M$. For $\epsilon=M\epsilon_x$, it holds that
	\begin{equation*}
		\sum_{n\in\N}P(\abs{X_nZ_n}>\epsilon \abs{a_nb_n})\leq n_0 +\sum_{n>n_0}P(\abs{X_n}>\epsilon_x \abs{a_n})<\infty.
	\end{equation*}
	
	(ii) The results follows using the same arguments as in (i) except that $\epsilon_z>0$ and $M>0$ can now be arbitrary.
	
	(iii) Suppose that $Z_n=\Oac(b_n)$. Then $\exists \epsilon_z>0$ such that $\sum_{n\in\N}\{P(\abs{X_n/a_n}>\epsilon_x/2)+P(\abs{Z_n/b_n}>\epsilon_z/2)\}<\infty$. For $\epsilon=2\max(\epsilon_x,\epsilon_z)$, we have that
	\begin{align*}
		\sum_{n\in\N}P\parent[\bigg]{\abs[\bigg]{\frac{X_n+Z_n}{a_n+b_n}}>\epsilon }&\leq 	\sum_{n\in\N}\chave[\bigg]{P\parent[\bigg]{\abs[\bigg]{\frac{X_n}{a_n+b_n}}>\frac{\epsilon }{2}}+P\parent[\bigg]{\abs[\bigg]{\frac{Z_n}{a_n+b_n}}>\frac{\epsilon }{2}}}\\
		&=\sum_{n\in\N}\chave[\bigg]{P\parent[\bigg]{\abs[\bigg]{\frac{X_n}{a_n}} \abs[\bigg]{\frac{1}{1+b_n/a_n}}>\frac{\epsilon }{2}}+P\parent[\bigg]{\abs[\bigg]{\frac{Z_n}{b_n}} \abs[\bigg]{\frac{1}{1+a_n/b_n}}>\frac{\epsilon }{2}}}\\
		&\leq \sum_{n\in\N}\chave[\bigg]{P\parent[\bigg]{\abs[\bigg]{\frac{X_n}{a_n}}>\frac{\epsilon }{2}}+P\parent[\bigg]{\abs[\bigg]{\frac{Z_n}{b_n}} >\frac{\epsilon }{2}}}<\infty.
	\end{align*}
	
	If $Z_n=O(b_n)$, then $\exists M>0,n_0\in\N:\forall n>n_0:\abs{Z_n/b_n}\leq M$. So, if we set $\epsilon=\epsilon_x+M$, then
	\begin{align*}
		\sum_{n\in\N}P\parent[\bigg]{\abs[\bigg]{\frac{X_n+Z_n}{a_n+b_n}}>\epsilon }&\leq n_0+\sum_{n>n_0}P\parent[\bigg]{\abs[\bigg]{\frac{X_n}{a_n}} \abs[\bigg]{\frac{1}{1+b_n/a_n}}+M\abs[\bigg]{\frac{1}{1+a_n/b_n}} >\epsilon}\\
		&\leq n_0+\sum_{n>n_0}P\parent[\bigg]{\abs[\bigg]{\frac{X_n}{a_n}}  >\epsilon-M}<\infty.
	\end{align*}
	
	(iv) The result follows by using the same arguments as in the proof of (iii) for $\epsilon_z,M>0$ arbitrary.
	
	(v) By hypothesis, for some $n_0,M>0$,
	\begin{equation*}
		\sum_{n\in\N}P(\abs{X_n/b_n}>\epsilon)\leq n_0+\sum_{n>n_0}P(\abs{X_n}>\epsilon\abs{a_n}/M)<\infty,
	\end{equation*}
	if we take $\epsilon=\epsilon_xM$.
	
	(vi) Let $\epsilon'>0$ be arbitrary.  By hypothesis, there is $n_0$ such that
	\begin{equation*}
		\sum_{n\in\N}P(\abs{X_n/b_n}>\epsilon)\leq n_0+\sum_{n>n_0}P(\abs{X_n}>\epsilon\abs{a_n}/\epsilon')<\infty,
	\end{equation*}
	if $\epsilon=\epsilon_x\epsilon'$.
	
	(vii) The result follows if the equivalence $X_n=\Oac(a_n)\iff X_n/a_n=\Oac(1)$ is satisfied. But it is trivially true from the definition.
\qed

\begin{prop}\label{propdan2}
	Let $a_n=o(1)$. For any $M\in\R$, it holds that
	\begin{equation*}
		X_n-M=\Oac(a_n)\implies 1/X_n=1/M+\Oac(a_n).
	\end{equation*}
\end{prop}

\noindent\textbf{Proof of Proposition \ref{propdan2}}
	\begin{align}\notag
		P\parent[\bigg]{\abs[\bigg]{\frac{1}{X_n}-\frac{1}{M}}>\epsilon\abs{a_n} }&\leq P(\abs{X_n-M}>\epsilon  \abs{a_nMX_n},\abs{X_n/M}\leq 2)+P(\abs{X_n/M}>2)\\\label{e9}
		&\leq  P(\abs{X_n-M}>2\epsilon M^2\abs{a_n})+P(\abs{X_n/M}>2)
	\end{align}
	
	By hypothesis, there exist $\epsilon'>0$ and $n_0\in\N$ such that 
	\begin{align*}
		\infty>\sum_{n\in\N}P(\abs{X_n-M}>\epsilon' \abs{a_n})\geq \sum_{n\in\N}P(\abs{X_n}>\epsilon' \abs{a_n}+\abs{M})\geq \sum_{n>n_0}P(\abs{X_n}>2\abs{M}).
	\end{align*}
	Hence, from \eqref{e9},
	\begin{align*}
		\sum_{n\in\N}	P\parent[\bigg]{\abs[\bigg]{\frac{1}{X_n}-\frac{1}{M}}>\epsilon\abs{a_n} }&\leq 2n_0+\sum_{n>n_0}\chave{P(\abs{X_n-M}>2\epsilon M^2\abs{a_n})\\
			&+P(\abs{X_n/M}>2)}<\infty
	\end{align*}
	for $\epsilon=\epsilon'/(2M^2)$.
\qed

\FloatBarrier
\bibliographystyle{apalike}
\bibliography{sup_biblio}